\pdfoutput=1
\documentclass[11pt]{article}
\usepackage{amsmath,amsfonts,amsthm,amssymb,subfiles}

\usepackage{fullpage}

\usepackage{graphicx, multicol}
\usepackage[colorlinks=true, allcolors=blue]{hyperref}

\makeatletter
\newtheorem*{rep@theorem}{\rep@title}
\newcommand{\newreptheorem}[2]{%
\newenvironment{rep#1}[1]{%
 \def\rep@title{#2 \ref{##1}}%
 \begin{rep@theorem}}%
 {\end{rep@theorem}}}
\makeatother

\newcommand{\maxroot}[1]{\mathrm{maxroot} \left\{ #1\right\}}

\newtheorem{thm}{Theorem}[section]
\newtheorem{cor}[thm]{Corollary}
\newtheorem{lemma}[thm]{Lemma}
\newtheorem{prop}[thm]{Proposition}
\newtheorem{claim}[thm]{Claim}

\newreptheorem{thm}{Theorem}
\newreptheorem{lemma}{Lemma}
\newreptheorem{prop}{Proposition}
\newreptheorem{cor}{Corollary}
\newreptheorem{claim}{Claim}

\theoremstyle{definition}
\newtheorem{defn}[thm]{Definition}
\newtheorem{rmk}[thm]{Remark}

\newreptheorem{defn}{Definition}
\newreptheorem{rmk}{Remark}

\newcommand{\Strans}[1]{\mathbb{S}#1}
\newcommand{\mysum}[2]{\boxplus_{#1}^{#2}}
\newcommand{\expect}[2]{\mathbb{E}_{#1} \left\{ {#2} \right\}}

\newcommand{\ii}{\mathbb{I}}
\newcommand{\deriv}[1]{\partial_{#1}}

\newcommand{\cauchy}[1]{\mathcal{G}_{#1}}

\newcommand{\Hcauchy}[1]{\mathcal{H}^{\lambda}_{#1}}
\newcommand{\Jcauchy}[1]{\mathcal{J}^{\lambda}_{#1}}
\newcommand{\Rtransform}[1]{\mathcal{R}_{#1}^{\lambda}}
\newcommand{\AND}{\quad\text{and}\quad}
\newcommand{\mydet}[1]{\det\left[#1\right]}
\newcommand{\mymatrix}[1]{
\begin{bmatrix}
y \ii_m & {#1} \\
{#1}^T & x \ii_d
\end{bmatrix}
}

\newcommand{\recsum}{\boxplus_{d,\lambda}}
\newcommand{\recsumfree}{\boxplus_\lambda}
\newcommand{\dd}{d} 

\newcommand{\mm}{m} 

\title{A theory of singular values for finite free probability}
\date{\today}
\author{Aurelien Gribinski\\ 
axgribinski@gmail.com\\
Princeton University
}
 
\begin{document}

\maketitle

\begin{abstract}
We introduce a finite version of free probability for rectangular matrices that amounts to operations on singular values of polynomials. This study is motivated by the companion papers \cite {MG1} and \cite{MG2}, as well as the corresponding paper dealing with the square case \cite{M}. In the process we exhibit a canonic bivariate operation on polynomials, seemingly more natural when singular values are concerned. We show that we can replicate the transforms from free probability, and that asymptotically there is convergence from rectangular finite free probability to rectangular free probability. Lastly, we show that classical distribution results such as a law of large numbers or a central limit theorem can be made explicit in this new framework where random variables are replaced by polynomials. \footnote{partially supported by ANR JCJC GALOP (ANR-17-CE40-0009)}
\end{abstract}

\section{Introduction}\label{sec:intro}
Free probability is a recent field created by Dan Voiculescu (see \cite{Speicherintro} or \cite{RS} and
references therein) that  studies operators on infinite dimensional spaces through the lens of moment distributions  and  convolutions - both from the combinatorial and analytic points of view. It led to many asymptotic results in random matrix theory. On the other hand, finite free probability was introduced in a series of papers by Marcus, Spielman and Srivastava (in particular \cite{MSS4}, \cite{MSS}) a few years ago, and proved some surprising analogues of results in free probability but for square matrices (operators on finite dimensional spaces) and their associated characteristic polynomials (analogues of moment distributions). They used their results to prove the existence of new large families of Ramanujan graphs. The theory was further developed some time after by Marcus  (\cite{M}). Basically, equalities in the realm of free probability turn experimentally into inequalities in the realm of finite free probability. This transition, however, is not well understood yet. The goal of this paper is to extend the theory of finite free probability in a systematic way from eigenvalues to singular values, or said otherwise from square hermitian matrices to rectangular matrices.

 \subsection{An introduction to rectangular free probability}

  In classical probability, if we are given two random variables $X_1$ and $X_2$ in two probability spaces $(M_1,\mu_1)$ and $(M_2,\mu_2)$, then one way to investigate the joint distribution  $\mu_{X_1,X_2}$ of $X_1$ and $X_2$  is to have access to the expectation  $\mathbb{E}[p(X_1,X_2)]$ for all bivariate polynomials $p$. In particular, one can calculate $\mu_{X_1,X_2}$ knowing only  $\mu_1$ and  $\mu_2$ if $X_1$ and $X_2$ are independent, in which case we have $\mu_{X_1,X_2}= \mu_1 \otimes \mu_2$. Another way to see it is that we know all product moments by the knowledge of the moments of $X_1$ and $X_2$ separately. \\
 In a similar spirit, we can extend these notions to noncommutative probability spaces. A noncommutative probability space $(\mathcal{A},\phi)$ is a unital algebra $\mathcal{A}$ over $\mathbb{C}$ and a unital linear functional (trace or expectation) $\phi: \mathcal{A} \rightarrow \mathbb{C}$  with $\phi(1)=1$. Elements of $\mathcal{A} $ are called noncommutative random variables. All the definitions and theorems as below can be found in \cite{RS}.
 \begin{defn}[distribution of a random variable] If there exists a probability distribution $\mu_a$ on $\mathbb{R}$ (we restrict to hermitian operators in the following) such that $\phi(a^k)= \int_{\mathbb{R}} z^k d\mu_a(z)$ for all $k \in \mathbb{N}$, then we call $\mu_a$ the (spectral) distribution of the noncommutative random variable $a$. In this situation, we call $a$ a spectral operator. 
 \end{defn}
 \begin{rmk}
 One can wonder if it is compatible with the spectrum of usual matrices when the expectation function is a trace. Consider the noncommutative unital algebra $M_d(\mathbb{C})$ of complex square matrices of size $d$,  on which there is  the normalized trace defined by $\phi(a)= \frac{1}{d}\sum \lambda_i(a)$, where the $\lambda_i$ are the eigenvalues of the matrix $a$.Then $\phi(a^k)= \frac{1}{d}\sum \lambda_i^k(a)$ and $\mu_a= \frac{1}{d}\sum \delta_{\lambda_i(a)}$, which  is the average of delta masses. The measure $\mu_a$  is indeed the eigenvalue distribution of $a$.\end{rmk}
 \begin{defn} [joint distribution for hermitian spaces]
 The joint distribution of spectral operators $a$ and $b$  consists of all crossed moments:$ \big\{ \phi(a_{1}a_{2} \dotsm a_{n}) | n \geq 1, a_{i} \in \{a,b\} \big\}$. 
 
 \end{defn}
 \begin{defn} [freeness in spectral case]
 We say that spectral operators $a$ and $b$ are free if for all $n \in \mathbb{N}$ and all univariate polynomials $p_1,p_2, \dots p_{2n}$ the following holds
 \[
 \phi[p_1(a)p_2(b)\dotsm p_{2n-1}(a)p_{2n}(b)]=0
 \]
 whenever $\phi[p_{2i-1}(a)]=\phi[p_{2i}(b)]=0 $  for  $ i \leq n$.
 These equalities give a way of computing recursively all crossed moments in terms of the moments of $a$ and $b$. 
 \end{defn}
 \begin{rmk}
The notion of freeness
gives calculation rules of mixed moments of random variables like the notion of classical
independence in probability theory.  Independence of random variables $X$ and $Y$  in the case where $\phi$ is the expectation is equivalent to
 \begin{align*}
 \phi[p_1(X)p_2(Y)]=0  &  \text{ whenever }  \phi[p_1(X)]= \phi[p_2(Y)]=0.
 \end{align*}
 Indeed, we have to notice that
 \[ \phi[(p_1(X)-\phi(p_1(X)))(p_2(Y)-\phi(p_2(Y)))]= \phi[p_1(X)p_2(Y)] -\phi[p_1(X)]\phi[p_2(Y)].
 \]
 The main difference between free independence and classical independence is that  free independence respects the noncommutativity of the variables. 
 \end{rmk}
 The most interesting  use of free independence for us is the analysis of the asymptotics of large random matrices. 
 \begin{defn}[Orthogonal invariance]
 Consider some hermitian random matrices $A_d$ (for $d \in \mathbb{N}$). We say that they form an orthogonally invariant random ensemble if they are invariant with respect to conjugation with a Haar orthogonal random matrix: that is the entries of $A_d$ and $Q_d^TA_dQ_d$ have the same joint distribution for a random $Q_d$ in the orthogonal group of size $d$ (for all $d \in \mathbb{N}$) . Conjugation by a Haar orthogonal random matrix corresponds to a random rotation of the eigenvectors. 
 \end{defn}
 The following theorem states  that independence of the eigenvectors (obtained by rotating randomly)  in the classical sense leads to the asymptotic freeness of the spectral distributions. 
 \begin{thm} [Voiculescu] 
Let $A_d$and $B_d$ be $d \times d$ real independent hermitian
random orthogonally invariant matrix ensembles whose empirical eigenvalue distributions of $A_d$ and
$B_d$ converge in probability to distributions of some spectral operators $a$ and $b$, respectively. Then $a$
and $b$ are freely independent.  \end{thm}
  
 \begin{defn}[free convolution]
 For freely independent spectral operators $a$ and $b$ with distributions $\mu_a$ and $\mu_b$ on $\mathbb{R}$ we denote by $ \mu_a \boxplus \mu_b $ the distribution of $a+b$. 
 \end{defn}
 \begin{rmk}\label{Voiculescu}
  If we take any independent random matrices $A_d$ and $B_d$ whose empirical eigenvalue distributions converge in probability to distributions $\mu_a$ and $\mu_b$ without any assumption on orthogonal invariance, then  $Q_d^TA_dQ_d$ and $\tilde{Q_d}^TB_d\tilde{Q_d}$ are orthogonally invariant (for Haar
distributed random orthogonal matrices $Q_d$ and $\tilde{Q_d}$) and we get that they are asymptotically free. As the empirical eigenvalue distributions of $Q_d^TA_dQ_d + \tilde{Q_d}^TB_d\tilde{Q_d}$ is the same as  $A_d + Q_d^TB_dQ_d$ (conjugation doesn't affect the eigenvalue distribution), we can conclude that the empirical eigenvalue distributions of $A_d + Q_d^TB_dQ_d$  is converging to $ \mu_a \boxplus \mu_b $. 
  \end{rmk}
 In order to compute the free convolution, we associate power series to distributions that behave well with respect to this operation.
 \begin{defn}[Cauchy and $R$-transforms] Define the Cauchy transform of a Borel measure $\mu$ on $\mathbb{R}$ as 
 \[
 \mathcal{G}_{\mu}(x):= \int_{t \in \mathbb{R}} \frac{d\mu(t)}{x-t}, \  \   \text{         for   Im}(x)>0.
 \]
  The $R$-transform is in turn defined as
 \[
 \mathcal{R}_{\mu}(x):= \cauchy{\mu}^{-1}(x) - \frac{1}{x}= \cauchy{\mu}^{-1}(x) -  \cauchy{\mu_0}^{-1}(x),
 \]
where $\mu_0$ is the Dirac mass at zero. By inverse we mean compositional inverse for power series (around $x=\infty$, that is $\cauchy{\mu}(x)$ is a power series in $\frac{1}{x}$). 
 \end{defn}
 \begin{rmk}
 If the measure $\mu$ is compactly supported, which is the case in this paper, then we can also define  $\mathcal{G}_{\mu}(x)$ for  real $x > \max\textup{Supp}\ \mu$, where $\textup{Supp}\ \mu$ is the support of $\mu$.
 
 \end{rmk}
The $R$-transform is useful because it fully characterizes the distribution and  linearizes the convolution:
\begin{thm} [linearization, see \cite{RS}] \label{linearization} For $x$ small enough (in the domain of convergence),
\[
\mathcal{R}_{\mu_a \boxplus \mu_b}(x)= \mathcal{R}_{\mu_a}(x)+ \mathcal{R}_{\mu_b}(x).
\]
\end{thm}	    
As the $R$-transform characterizes the distribution, the previous theorem allows us to compute the free sum distribution of two operators. 
 \\

We can extend the definition of freeness for general (non spectral) operators that admit an adjoint.The idea is that an operator $a$ that admits an adjoint $a^{\star}$ is such that $a^{\star}a$ is spectral.

\begin{defn} [joint distribution for $\star$-probability spaces]
 The notion of joint distributions for hermitian operators can be generalized to a similar one for all operators that admit some adjoint.
The joint distribution of $a$ and $b$ consists in all crossed moments:\\
$ \big\{ \phi(a_{1},a_{2} \dotsm a_{n}) |  n \geq 1, a_{i} \in \{a,b,a^{\star},b^{\star}\} \big\}$ where $a^{\star}$ represents the adjoint. 
\end{defn}

 \begin{defn}[freeness for $\star$-probability spaces]
 We say that $a$ and $b$ are free if for all $n$ and all bivariate polynomials $p_1,p_2,\dots,p_{2n}$,
  \[
 \phi[p_1(a,a^{\star})p_2(b,b^{\star})\dotsm p_{2n-1}(a,a^{\star})p_{2n}(b,b^{\star})]=0
 \]
 whenever $\phi[p_{2i-1}(a,a^{\star})]=\phi[p_{2i}(b,b^{\star})]=0 $  for all  $ i \leq n$.
 These equalities give a way to compute recursively all crossed moments in terms of the moments of $a$ and $b$.

 We wish to connect this extended notion of freeness to random matrices.  Let's recall that if a matrix $M$ of size  $m \times 
d$ has singular decomposition $M= UDV$, with $U$ and $V$ orthogonal matrices of size $\mm \times \mm$ 
and $\dd \times \dd$, and $D$ is a nonnegative diagonal matrix of size   $m \times d$, then the 
singular values are the elements of the diagonal of $D$. They also correspond 
to the square roots of the eigenvalues of $M^{T}M$. 
The uniform 
distribution on the singular values will be called the \textit{singular law} of $M$. A 
random matrix is said to be \textit{bi-orthogonally invariant} if its distribution is 
invariant under the left and right actions of the orthogonal group. For a 
probability measure $\mu$  on $\mathbb{R}$, denote by $\tilde{\mu}$ the symmetrization 
of $\mu$, which is the probability measure defined by $\tilde{\mu}(B) 
=\frac{\mu(B) +\mu(-B)}{2}$ for all Borel sets $B$.

\end{defn}
A first step was 
accomplished by Voiculescu \cite{V} who proved the following.
 
 \begin{thm} [free square singular addition]
The asymptotic singular law of the sum of two independent, bi-orthogonally invariant random
square matrices such that the symmetrizations of the respective singular laws 
converge weakly to the probability measures $\mu_1$ and $\mu_2$, respectively, only depends on 
$\mu_1$ and $\mu_2$, and can be expressed easily from $\mu_1$ and $\mu_2$: it 
is the probability measure on $[0,\infty)$, the symmetrization of which is the 
free convolution of $\mu_1$ and $\mu_2$. Notice that if we call the symmetrization of the limiting distribution of the sum  $\mu_1 \mysum{}{1}\mu_2$, then  $\mu_1 \mysum{}{1}\mu_2=  \mu_1 \mysum{}{}\mu_2$. Furthermore, the random matrices become free in the limit. 

\end{thm}
  \begin{rmk} \label{asymptoticssingularV}
  If $A_d$ and $B_d$ are two independent square random matrix ensembles and $Q_d$, $R_d$, $\tilde{Q_d}$, $\tilde{R_d}$ are Haar random orthogonal matrices of size $d$, then $Q_d^{T}A_dR_d$ and $\tilde{Q_d}^{T}B_d\tilde{R_d}$ are free in the limit $d$ going to infinity, and we  know the asymptotic singular law of $Q_d^{T}A_dR_d+\tilde{Q_d}^{T}B_d\tilde{R_d}$. But as the singular distribution is unchanged by left and right conjugation, then it is the same as the singular law of $A_d+Q_d^{T}B_dR_d$. 
  \end{rmk}
 
  It was then generalized by Benaych-Georges to 
rectangular matrices (\cite{BG}). However it should be noted that in this case rectangular matrices don't form an algebra, therefore the notion of freeness has to be redefined to deal with delicate undefined products. In this regard,  a notion of freeness with amalgamation is necessary. 
\begin{thm}[free rectangular singular addition from \cite{BG}]\label{asymptoticssingularBG}
Let, for all $n \geq1$,$M(1,n)$ and $M(2,n)$ be independent bi-orthogonally 
invariant $q_1(n) \times q_2(n)$ random matrices with $q_1(n) \geq q_2(n)$, and such that for all i= 1,2, 
the symmetrization of the singular law of $M(i,n)$ converges in probability to 
$\mu_i$. Then the symmetrization of the singular law of $M(1,n)+M(2,n)$ 
converges in probability to a symmetric probability measure on the real line, 
denoted by $\mu_1\mysum{}{\lambda}\mu_2$, which depends only on $\mu_1,\mu_2$,
and $\lambda:= \lim_{n \rightarrow \infty}q_2(n)/q_1(n)$. Notice that $\lambda \in [0,1]$.
\end{thm}

 The tool that plays the role of the Cauchy transform in the rectangular setting is the following quadratic transform. 
 \begin{defn}[from \cite{BG}] \label{defnHtransform} The $ \lambda$-rectangular Cauchy transform of a symmetric compact measure $\mu$ (and $x$ in a nonnegative neighborhood of $0$) is given by
\[
H^{\lambda}_{\mu} (x)= \lambda \big[\mathcal{G}_{\mu} (\frac{1}{\sqrt{x}}) \big]^2 +(1-\lambda)\sqrt{x}\mathcal{G}_{\mu} (\frac{1}{\sqrt{x}}) = x+ \sum_{i=2}^{\infty} h_i^{(p)} x^i,
\]
where the $h_i^{(p)}$ are the coefficients we get in the expansion.
\end{defn}

\begin{lemma} [from \cite{BG}]  As the measure is compact, $H^{\lambda}_{\mu} (0)=0$ and $ \frac{d H^{\lambda}_{\mu}}{dx}(0)=1$ , the rectangular transform is analytic in a neighborhood of  $x=0$, and therefore admits a compositional inverse which is also analytic  in a neighborhood of zero.This inverse will be denoted by $ [{H_{\mu}^\lambda}]^{-1}(x)$.
\end{lemma}

\begin{defn}\label{Rrectdefn} [following \cite{BG}] For $x$ small enough, let
\[
U^{\lambda}(x):= \frac{-\lambda -1 +\big[ (\lambda+1)^2+4\lambda x\big]^{1/2}}{2\lambda}.
\]
The rectangular  $R$-transform is given by
\[
\mathcal{R}_{\mu}^{\lambda}(x):= U^{\lambda} \Big ( \frac{x}{[{H_{\mu}^\lambda}]^{-1}(x)} -1\Big).
\]
\end{defn}

\begin{rmk}
$\mathcal{R}_{\mu}^{\lambda}(x)$ is also analytic in a neighborhood of zero by theorems of composition given that the square root is well-defined and $\frac{x}{[{H_{\mu}^\lambda}]^{-1}(x)} $ can be expanded at zero (because it doesn't vanish).
\end{rmk}
 
 \begin{thm}[from \cite{BG}]
 The rectangular $R$-transform linearizes the rectangular additive convolution for symmetric measures $\mu_1$ and $\mu_2$,that is, for $x$ small enough (in the domain of convergence):
 \[
 \mathcal{R}_{\mu_1 \mysum{}{\lambda}\mu_2}^{\lambda}(x) = \mathcal{R}_{\mu_1}^{\lambda}(x)+\mathcal{R}_{\mu_2}^{\lambda}(x).
 \]
 \end{thm}
 
 \subsection{New results and organization of the paper}
 
 \subsubsection{A new polynomial operation on singular values of rectangular matrices}
 We start by introducing in Section~\ref{bivariatedef} a new operation on polynomials with nonnegative roots, or alternatively, on even realrooted polynomials. We define the symmetrization operator: $\Strans{p}$ which denotes for a 
polynomial $p$  with all real nonnegative roots the polynomial $p(x^2)$ (its roots are 
the square-roots of $p$ plus the symmetric negative numbers). So being given the polynomial $p$ with nonnegative roots or the even realrooted polynomial $\Strans{p}$ is the same. In all this paper, we 
will consider a polynomial and the uniform measure $\mu_p$ over its roots as 
giving the same information. 
 If $p:=\prod_i(x-\lambda_i(p))$, and $\lambda_i(p)\geq 0$ for all $i$, we associate to $p$ the measure
\[ \mu_{\Strans{p}}=\frac{1}{2d}\sum_{i=1}^{d}( \delta_{\sqrt{\lambda_i(p)}} + \delta_{-\sqrt{\lambda_i(p)}}).
\] 

 Consider two rectangular matrices $A$ and $B$  of size $m\times d$ where $m\geq d$,  with symmetrized singular distributions $\mu_A$ and $\mu_B$, and   orthogonal matrices $Q$ in $\mathcal{O}_m$  and $R$ in $\mathcal{O}_d$. Free probability (see Theorem~\ref{asymptoticssingularBG}) tells us that the symmetrization of the singular distribution of $A+Q^TBR$  is close to the Benaych-Georges' rectangular free sum $\mu_A \mysum{}{\frac{d}{m}} \mu_B$ when $d$, $m$ are large (for $\lambda= \frac{d}{m}$ fixed)--- that is, the roots of $\chi_{(A+QBR^T)(A+QBR^T)^T}$ can be predicted with a good accuracy. One possible way to create a deterministic finite distribution on the model of free probability that doesn't depend on the instances $A$ and $B$ but only on the distributions is to look at all possible characteristic polynomials $\chi_{(A+QBR^T)(A+QBR^T)^T} $ and average them uniformly. In the limit, as it converges to the same distribution for all random matrices, it will also heuristically converge to this same distribution when we take the expectation. This led us to define the following univariate convolution in \cite{MG1}:
 \begin{defn}[Rectangular singular free sum] \label{rectconv:defn} For $m \times d$ rectangular matrices A and B, define
 \begin{align*}
 \chi_{A^TA}  \mysum{d}{m-d} \chi_{B^TB} &:= \expect{R \in \mathcal{O}_m, Q \in \mathcal{O}_d}{\chi_{(A+QBR^T)(A+QBR^T)^T}}\\
 &= \iint _{\mathcal{O}_m\times \mathcal{O}_d}  \mydet{xI - (A+QBR^T)^T(A+QBR^T)} dR dQ
 \end{align*}
 where the measures are Haar on the respective orthogonal groups. 
 \end{defn}
 
 We derived a binomial formula that enabled us to extend the convolution to polynomials of degree at most $d$:
 \begin{thm}[From \cite{MG1}]
 Consider two polynomials $p$ and $q$ with only real nonnegative roots (they can be written as $ \chi_{A^TA} $ and  $\chi_{B^TB}$ for some $m\times d$ matrices $A$ and $B$). If  we write $p(x)= \sum_{i=0}^d (-1)^ia_ix^{d-i}$ and $q(x)= \sum_{i=0}^d (-1)^i b_ix^{d-i}$ the following holds
\begin{align*}
p  \mysum{d}{m-d} q = \sum_{k=0}^d x^{d-k}(-1)^k \sum_{i+j=k} \frac{(d-i)!(d-j)!}{d!(d-k)!}\frac{(m-i)!(m-j)!}{m!(m-k)!}a_ib_j
\end{align*}

 \end{thm}
 We also proved that such a convolution is real-rooted with nonnegative roots:
  \begin{thm}[Realrootedness of the convolution, from \cite{MG1}]\label{thm:rr}
 For $p$,$q$ polynomials with nonnegative real roots of degrees at most than $d$, the polynomial we get through the convolution is real rooted with nonnegative roots (the operation $\mysum{d}{m-d}$ is stable over polynomials with nonnegative real roots). Furthermore, the operation $\mysum{d}{m-d}$ is associative and bilinear. 
 \end{thm}
In this paper, we extend the definition to special bivariate polynomials.
\begin{defn}
Fix $m$ and $d$. 
 If $p$ is a realrooted polynomial with nonnegative real roots of degree $d$, define the "rectangular" polynomial extension of order $m-d$ as $p(x,y):= y^{m-d}p(xy)$. 
\end{defn}

\begin{rmk}
If for a $m\times d$ matrix $A$, $p(x)=  \chi_{A^TA}$ then 
\[
p(x, y) 
= \det \mymatrix{A}.
\]
\end{rmk}
Now given that if $p$ has only nonnegative roots, such a matrix $A$ can always be exhibited whatever the value of $m \geq d$, we are led to the following generalization:

 \begin{repdefn}{rectconvdefn}[Rectangular bivariate convolution]
 If
\[
p(x, y) 
= \det \mymatrix{A}
\AND
q(x, y) 
= \det \mymatrix{B},
\] 
then for $\lambda:=\frac{d}{m}$ (for a given $d$, $m$ and $\lambda$ are in bijection), define the $\lambda$ -{\em rectangular additive convolution} of $p$ and $q$ as
\begin{align*}
[p \recsum q](x, y) 
&:= \iint 
\mydet{
\mymatrix{A} + 
\begin{bmatrix}
\hat{Q} & 0  \\
 0 & \hat{R}
\end{bmatrix}
\mymatrix{B}
\begin{bmatrix}
\hat{Q} & 0  \\
 0 & \hat{R}
\end{bmatrix}^T
}
d\hat{Q} d\hat{R}
\\&= \iint
\det \mymatrix{(A + \hat{Q} B \hat{R})}
  d\hat{Q} d\hat{R},
\end{align*}
where $d\hat{Q}, d\hat{R}$ can be taken to be the Haar measure over the orthogonal group (it could be made more general but we won't extend it for the sake of simplicity).
\end{repdefn}

 \subsubsection{Defining a transform fit for this polynomial convolution}
In Section~\ref{defnRfinite}, we  associate bijectively to every polynomial of degree $d$ with nonnegative real roots $p$ its finite $R$-transform $\mathcal{R}^{d,\lambda}_{\Strans{p}}(s) $ as another polynomial of degree $d$. We prove the following crucial linearization emulating property:
\begin{repthm}{Rtransform:additivity}
For two polynomials $p$ and $q$ with nonnegative real roots of degree $d$, we have
\[
\mathcal{R}^{d,\lambda}_{\Strans{[p \boxplus_{d,\lambda} q]}}(s) = \mathcal{R}^{d,\lambda}_{\Strans{p}}(s) +\mathcal{R}^{d,\lambda}_{\Strans{q}}(s) .
\]
\end{repthm}
It is the direct analogue of the free probability additivity property that defines the free $R$-rectangular transform:
\begin{align*}
 \Rtransform{\mu_{\Strans{p}} \recsumfree \mu_{\Strans{q}}} (s) &=  \Rtransform{\mu_{\Strans{p}}} (s) + \Rtransform{\mu_{\Strans{q}}} (s) .
\end{align*}
 
 \subsubsection{Adapting free probability notions to connect them to polynomials}
 We derive in Section~\ref{adapt} more simple formulas for discrete measures. In the following definition we define the $\lambda$-rectangular Cauchy transform of a polynomial $p$  that is not exactly the $\lambda$-rectangular Cauchy transform of a measure $\mu$ from the Definition~\ref{defnHtransform}.
\begin{repdefn}{Htransform:defn}
For a polynomial $p$, the $\lambda$-rectangular Cauchy transform of its symmetrized polynomial $\Strans{p}$ with all roots nonnegative is given by
\[
\Hcauchy{\Strans{p}}(x):= \mathcal{G}_{\mathbb{S}p}(x) \Big(  \lambda \mathcal{G}_{\mathbb{S}p}(x)+ (1-\lambda) {\mathcal{G}}_0(x)  \Big)  \textit{\ \ \ for \ } x> \sqrt{\maxroot{p}}.
\]
\end{repdefn}
Recall that ${\mathcal{G}}_0(x) = \frac{1}{x}$. 
It is easy to check that
\[
\Hcauchy{\Strans{p}}(x)=H^{\lambda}_{\mu_{\Strans{p}}} (\frac{1}{x^2}).
\]
This definition is a slightly modified version of the Benaych Georges' transform fit for polynomials. Our definition also incorporates the symmetrization. 
\begin{replemma} {Hinverse:defn} $\Hcauchy{\Strans{p}}$ is a bijection from $[\sqrt{\maxroot{p}}, +\infty]$ to $ [0,+\infty]$. We denote the inverse by $\Jcauchy{\Strans{p}}$. In particular, for $x> 0$,
\[
\Hcauchy{\Strans{p}} \circ \Jcauchy{\Strans{p}} [x]=x,
\]
and for  $ x>\sqrt{\maxroot{p}}$,
\[
\Jcauchy{\Strans{p}} \circ \Hcauchy{\Strans{p}} [x]=x.
\]
\end{replemma}

We can rewrite the $R$-transform for polynomials:
\begin{replemma}{Rtransformform}
\[
\Rtransform{\Strans{p}}(s^2):= \Rtransform{\mu_{\Strans{p}}}(s^2)= \frac{-\lambda -1}{2\lambda} +\sqrt{ \frac{(\lambda-1)^2}{4\lambda^2}+  \frac{ s^2 [\Jcauchy{\Strans{p}}(s^2)]^2 }{\lambda} }.
\]

\end{replemma}

\subsubsection{Defining a modified finite $R$-transform}
We define in Section~\ref{modifiedtransform} an alternative power series  $\widetilde{\mathcal{R}}^{d,\lambda}_{\Strans{p}}(s) $ that converges by design to $ \Rtransform{\Strans{p}} (s^2)$ using approximations of $\Jcauchy{\Strans{p}}(s^2)$.

\begin{repthm}{Convmodified}[Convergence of the modified finite $R$-transform to the free $R$-transform] 
For $s$ small enough,
\[
\widetilde{\mathcal{R}}^{d,\lambda}_{\Strans{p}}(s) \xrightarrow{d\to \infty}    \Rtransform{\Strans{p}} (s^2).
\]
\end{repthm}

 Secondly, we show it that has the same first $d/2$ coefficients as  $\mathcal{R}^{d,\lambda}_{\Strans{p}}(s)$:
\begin{repthm}{partialequality}
 \[
 \widetilde{\mathcal{R}}^{d,\lambda}_{\Strans{p}}(s)  \equiv \mathcal{R}^{d,\lambda}_{\Strans{p}}(s^2)   \mod [ s^{d+1}].
 \]
 \end{repthm}

\subsubsection{Convergence: from finite free probability to free probability}

In Section~\ref{sec:convergence}, we gather all the pieces to prove the main theorem of this paper. 
 \begin{repthm}{convergence} Coefficientwise or for $s$ small enough, the following power series convergence holds:
 \[
\mathcal{R}^{d,\lambda}_{\Strans{p}}(s)  \xrightarrow{d\to \infty}  \Rtransform{\Strans{p}} (s).
\]
By coefficientwise convergence we mean that each sequence of coefficients of  $\mathcal{R}^{d,\lambda}_{\mu_A}(s)$ converges to the corresponding coefficient of  $\mathcal{R}^{\lambda}_{\mu_A}(s)$, which is a combinatorial statement. Pointwise convergence is stronger and means that the overall power series in $s$ converges  (as $d$ goes to infinity) to the asymptotic power series; it is an analytic statement. 
 \end{repthm}

\subsubsection{Polynomials as random variables: limit theorems}
Finally, in Section~\ref{limittheorems}, we show that the connexion between probability and polynomials goes actually further. We consider polynomials of the form $\Strans{p}$ as random variables in our rectangular polynomial framework. 
For $p$ of degree $d$, and $\alpha>0$, define the renormalizing operator on the roots $\mathit{R_{\alpha}}(p):= \alpha^{-d}p(\alpha x)$.
$\Strans{q}=x^{2d}$ is the zero polynomial in terms of rectangular addition in the sense that $\Strans{[p \boxplus_{d,\lambda} q]}=  \Strans{p}$. 
For two polynomials $p$ and $q$, we will use $p \approx q$ to express that they have the same roots but not the same leading coefficient.
 \begin{repprop}{lawlargenumbers}[Law of large numbers] Let $p_1,p_2$,... be a sequence of degree $d$ polynomials with real nonnegative roots and means uniformly bounded by $\sigma^2$, that is,
\begin{align*}  & p_i:= \prod_j (x-r_{i,j}^2) &\text{and}&&  \frac{1}{d}\sum_j r_{i,j}^2 \leq \sigma^2 &
\end{align*}  
then
\[
\lim_{N \to \infty} \mathit{R_{1/{N}}}(\Strans{[p_1\boxplus_{d,\lambda}... \boxplus_{d,\lambda} p_N]}) (x)\approx  x^{2d}.
\]
\end{repprop}
 
\begin{replemma}{laguerrecumulant}[Laguerre polynomials]
For p of degree $d$ with nonnegative real roots, and $\sigma^2>0$, the following are equivalent:
\begin{enumerate}
\item $\mathcal{R}^{d,\lambda}_{\Strans{p}} (s) =m\sigma^2 s$, so only the first nontrivial cumulant is nonzero.
\item $ p \approx L_d^{(m-d)}(\frac{x}{\sigma^2})$, that is $p$ is up to scaling a generalized Laguerre polynomial of parameter $m-d$.
\end{enumerate}
\end{replemma}
Therefore Laguerre polynomials play the role of the Gaussians in our framework, which is made clearer by the following. 
\begin{repprop}{centrallimittheorem}[Central limit theorem] Let $p_1,p_2$,... be a sequence of degree $d$ with real nonnegative roots and same mean $\sigma^2$, that is,
\begin{align*}  & p_i= \prod_j (x-r_{i,j}^2) &\text{and}&&  \frac{1}{d}\sum_j r_{i,j}^2=\sigma^2 &
\end{align*} 
then
\[
\lim_{N \to \infty} \mathit{R_{1/\sqrt{N}}}(\Strans{[p_1\boxplus_{d,\lambda}... \boxplus_{d,\lambda} p_N]}) (x)\approx  L_d^{(m-d)}(\frac{x^2m}{\sigma^2}).
\]
\end{repprop}

\section{Bivariate definition of the convolution} \label{bivariatedef}
We first need to generalize the usual characteristic polynomial in order to store information not only about the singular values but also about the dimensions. 
The main issue is that there really are two different characteristic polynomials associated with a given set of singular values, that is, if $A$ is an $m \times d$ matrix with $ m \geq d$,
\[
\psi_A(x, 1) := \det[x \ii_d- A^TA]  
\]     
and
\[
\psi_A(1, x) := \det[x \ii_m- AA^T] .
\]
The first one contains all the singular values but fails to give information about the additional rectangular dimension, $m$. The second gives  $m$ but not $d$, and adds a bunch of zeros to the singular values. We need to keep track of both dimensions and in this respect, considering a bivariate polynomial works well.
 We consider
\begin{align*}
\psi_A(x, y) 
&= \det \mymatrix{A}
\\&= y^{m-d} \mydet{ x y \ii_{d} - A^TA}\\
&= \left(\frac{1}{x}\right)^{m-d} \mydet{ x y \ii_{m} - AA^T}.
\end{align*}
Note that although this is ``technically'' a bivariate polynomial, it is in 
some sense a univariate polynomial with an extra dimension term added on.  
Now, let $A, B$ be $m \times d$ matrices and let their generalized (singular) 
characteristic polynomials be:
\[
p(x, y) 
= \det \mymatrix{A}
\AND
q(x, y) 
= \det \mymatrix{B}.
\] 
 \begin{defn}\label{rectconvdefn}
For $\lambda:=\frac{d}{m}$ (for a given $d$, $m$ and $\lambda$ are in bijection), define the $\lambda$ -{\em rectangular additive convolution} of $p$ and $q$ as
\begin{align*}
[p \recsum q](x, y) 
&:= \iint 
\mydet{
\mymatrix{A} + 
\begin{bmatrix}
\hat{Q} & 0  \\
 0 & \hat{R}
\end{bmatrix}
\mymatrix{B}
\begin{bmatrix}
\hat{Q} & 0  \\
 0 & \hat{R}
\end{bmatrix}^T
}
d\hat{Q} d\hat{R}
\\&= \iint
\det \mymatrix{(A + \hat{Q} B \hat{R})}
  d\hat{Q} d\hat{R}
\end{align*}
where $d\hat{Q}, d\hat{R}$ can be taken to be the Haar measure over the orthogonal group (it could be made more general but we won't extend it for the sake of simplicity).
\end{defn}
\begin{rmk}\label{rectconvspec}
We get the univariate rectangular convolution by plugging $y=1$:
\[
[p \recsum q](x, 1)= p \mysum{d}{m-d}q=  \iint \mydet{  x \ii - (A + \hat{Q} B \hat{R})^T (A + \hat{Q} B \hat{R}))  } d\hat{Q} d\hat{R}.
\]
 In this paper, we will denote by $\boxplus_{d,\lambda}$ the convolution $\mysum{d}{m-d}$ for $\lambda=\frac{d}{m}$ even when restricting to the univariate case. This will be convenient as the parameter that is constant will be $\lambda$ and not $m-d$. 
\end{rmk}
\begin{prop}\label{rot}
Let $A, B$ be $m \times d$ matrices and let
\[
p(x, y) 
= \det \mymatrix{A}
\AND
q(x, y) 
= \det \mymatrix{B}.
\] 

There exists an $m \times d$ matrix $C$ such that
\[
[p \recsum q](x, y) = \det \mymatrix{C}.
\]
\end{prop}

\begin{proof}
Using Theorem \ref{thm:rr}, we get that $ [p \recsum q](x, 1) $ has nonnegative roots. Therefore we can take square-roots of such values, and we choose $C$ to be a matrix such that it's singular values are these square roots. We get
\[
[p \recsum q](x, 1) = \det(x \ii_d- C^TC ).
\]
It follows that
\[
[p \recsum q](x, y) = y^{m-d}\det(xy \ii_d- C^TC ).
\]

\end{proof}
It  defines a ``rotation invariant'' operation on the algebra of singular 
values. To push the analogy with the hermitian case (see \cite{MSS}), we can also exhibit a derivative expression.

\begin{lemma} \label{bivsum}
Consider polynomials $p$ and $q$ with nonnegative real roots and $p(x,y)=y^{m-d}p(xy), \\
q(x,y)=y^{m-d}q(xy)$ their bivariate extensions, then
\[
[p \recsum q ](x,y)= \frac{(m-d)!}{d! m!}\sum_{k=0}^{d} (\deriv{x}\deriv{y})^{d-k}p(x,y)  (\deriv{x}\deriv{y})^{k}q(0,1),
\]
\[
[p \mysum{d}{m-d}q ](x)= \frac{(m-d)!}{d! m!}\sum_{k=0}^{d} (\deriv{x}\deriv{y})^{d-k}p(x,1)  (\deriv{x}\deriv{y})^{k}q(0,1).
\]
\end{lemma}
\begin{proof}
By Remark~\ref{rectconvspec} and Proposition~\ref{rot}, one can see that
\[
[p \recsum q] (x,y)= y^{m-d}[p \recsum q] (xy,1)= y^{m-d} [p \boxplus_{d}^{m-d} q] (xy).
\]
It is sufficient to check the equality for $p(x,y)= y^{m-d} (xy)^{d-i}$ and  $q(x,y)= y^{m-d} (xy)^{d-j}$  due to the bilinearity. By definition of $\recsum $,
we have:
\[
y^{m-d} (xy)^{d-i} \recsum  y^{m-d} (xy)^{d-j} =  y^{m-d} (xy)^{d-i-j}  \frac{(d-i)!(d-j)! }{d! (d-i-j)!} \frac{(m-i)!(m-j)! }{m! (m-i-j)!}.
\]
On the other hand, a simple computation shows that  $(\deriv{x}\deriv{y})^{k}q(0,1) \neq 0 \iff k=d-j$ and therefore
\[
 \frac{(m-d)!}{d! m!}\sum_{k=0}^{d} (\deriv{x}\deriv{y})^{d-k}p(x,y)  (\deriv{x}\deriv{y})^{k}q(0,1) =  \frac{(m-d)!}{d! m!} (\deriv{x}\deriv{y})^{j}p(x,y)  (\deriv{x}\deriv{y})^{d-j}q(0,1).
\]
Here we get
\begin{align*}
 (\deriv{x}\deriv{y})^{j}p(x,y) &= \frac{(m-i)!}{(m-i-j)!} y^{m-i-j}  \frac{(d-i)!}{(d-i-j)!} x^{d-i-j} \\  
 \text{and}\\
  (\deriv{x}\deriv{y})^{d-j}q(0,1) &= (d-j)! \frac{(m-j)!}{(m-d)!}.
 \end{align*}
 We conclude by gathering terms. 
 \end{proof}

\begin{lemma}\label{diffop:sum}
Let $\mathcal{L} = \sum_i c_i (\deriv{x} \deriv{y})^i$ be a linear differential 
operator.
Then we can pull out the differential operator from the convolution operation:
\[
\mathcal{L}\left\{ p \recsum q \right\} = \mathcal{L}\{p\} \recsum q = p 
\recsum \mathcal{L}\{q\}.
\]
\end{lemma}

So if $\mathcal{P}$ and $\mathcal{Q}$ are linear differential operators such 
that
\[
p(x, y) = \mathcal{P}\{ y^m x^d \}
\AND	
q(x, y) = \mathcal{Q}\{ y^m x^d \}
\]
then
\begin{align*} 
[p \recsum q]
&= [\mathcal{P} \{ y^m x^d \} \recsum \mathcal{Q} \{ y^m x^d \}]
\\ &= \mathcal{P} \{ \mathcal{Q} [y^m x^d] \recsum y^m x^d  \}
\\ &= \mathcal{P} \mathcal{Q} \{  y^m x^d \}.
\end{align*}
\begin{proof}
Using Lemma \ref{bivsum}, we see that the operation $\recsum$ commutes with $\deriv{x} \deriv{y}$ and the result follows.
\end{proof}

\section{Definition of the rectangular finite $R$-transform for bivariate polynomials}\label{defnRfinite}
\subsection{From polynomials to random variables: the rectangular $T$-Transform}
We recall that the symmetrization operator $\Strans{p}$  denotes for a 
polynomial $p$  with all real nonnegative roots the polynomial $p(x^2)$. We 
will consider a polynomial and the uniform measure $\mu_p$ over its roots as 
giving the same information. 
 If $p:=\prod_i(x-\lambda_i(p))$, and $\lambda_i(p)\geq 0$ for all $i$, we associate to $p$ the measure
\[ \mu_{\Strans{p}}=\frac{1}{2d}\sum_{i=1}^{d}( \delta_{\sqrt{\lambda_i(p)}} + \delta_{-\sqrt{\lambda_i(p)}}).
\] 

Consider a symmetric discrete finite probability measure $\mu_A$  such that $\mu_A= \mu_{\Strans{p_A}}$, for a polynomial $p_A $ with nonnegative roots. If such a polynomial  $p_{A}$ of degree $d$ exists, then it is the unique polynomial of degree $d$ such that this holds. Now fix $\lambda \in ]0,1]$. For each $\lambda$, and each $d$ such that $ d=m\lambda$ (for an integer $m \geq d$), we can uniquely associate a random variable to the measure $\mu_A$. We call it the $T^{(m,d)}$-transform or also the $T^{(d,\lambda)}$-transform.  Alternatively, for a given polynomial of degree $d$ with nonnegative real roots $p_A$, we associate uniquely a $T^{(m,d)}$-transform. In the following, the expectation symbol  will stand for a uniform mean over all the values taken by the random variable ($T$-transform).
\begin{defn}\label{defT}
For $d$, $m$, $\lambda$, $\mu_A$ and $p_A$ as above, we define the random 
variable $T^{(m,d)}_{\Strans{p_A}}$ or $T^{(m,d)}_{\mu_A}$ , associated to the measure $\mu_A$ and vice versa by
\[
  \mathbb{E}\big[ e^{-T^{(m,d)}_{\mu_A} \partial_x \partial_y} \big] \{y^{m}x^{d}\}=\frac{1}{d} \sum_{i=1}^d e^{-t_i^{(m,d)}(A)[\partial_x \partial_y]}\{y^{m}x^{d}\}:=y^{m-d}p_{A}(xy)
\]
where $T^{(m,d)}_{\Strans{p_A}}$, the rectangular $T$-transform associated to $p_A$,  is a random variable taking with same probability any value from the multiset $({t_1^{(m,d)}(A),....,t_d^{(m,d)}(A) )}$.
\end{defn}

 \begin{claim}\label{unic}
 $T^{(m,d)}_{\Strans{\mu_A}}$ is well-defined and uniquely determines $\mu_A$ (or $p_A$).
 \end{claim}
 \begin{proof}
 Write:
 \[
 y^{m-d}p_{A}(xy)= y^{m-d}\sum_{i=1}^d (xy)^{d-i}(-1)^i p_i.
 \]
 We get:
 \begin{align}
  \mathbb{E}\big[ e^{- T^{(m,d)}_{\mu_A}\partial_x \partial_y} \big] \{y^{m}x^{d}\}&= \sum_{i=0}^d \frac{\mathbb{E}\big[(-T^{(m,d)}_{\mu_A})^i \big]}{i!}  (\partial_x \partial_y)^{i} \{y^{m}x^{d}  \}\\
  &= \sum_{i=0}^d \frac{\mathbb{E}\big[(-T^{(m,d)}_{\mu_A})^i \big]}{i!} \frac{m!}{(m-i)!}\frac{d!}{(d-i)!}y^{m-i}x^{d-i}\\
  &= y^{m-d}\sum_{i=0}^d (xy)^{d-i} (-1)^i\frac{\mathbb{E}\big[(T^{(m,d)}_{\mu_A})^i \big]}{i!} \frac{m!}{(m-i)!}\frac{d!}{(d-i)!}.
 \end{align}
 So we have to prove the existence and the unicity of $T^{(m,d)}_{\mu_A}$ such that for all $i \in \{1,\dots,d\}$:
 \[
 \mathbb{E}\big[(T^{(m,d)}_{\mu_A})^i \big] =  \frac{i!(m-i)!}{m!}\frac{(d-i)!}{d!} p_i,
 \]
 or
  \[
 \sum_{j=1}^{d}[t_j^{(m,d)}(A)]^i=  d\frac{i!(m-i)!}{m!}\frac{(d-i)!}{d!} p_i.
 \]
 It amounts to finding $d$ complex numbers for which  the $d$ first power sums are given. The existence and the unicity follows by using Newton's identities. 
 \end{proof}
  \begin{rmk}
 Notice that the numbers $t_j^{(m,d)}(A)$ are \textit{a priori} complex. But the expectations $\mathbb{E}\big[(-T^{(m,d)}_{\mu_A})^i]$ are real. 
 \end{rmk}
Now, let's see why this representation turns out to be useful in the framework of finite free probability.

\begin{prop}\label{corT}

For any multisets $T^{(m,d)}_{\Strans{p}}$ and  $T^{(m,d)}_{\Strans{q}}$associated to polynomials $p$ and $q$ with nonnegative roots and degree $d$, considered by assumption independent (we choose the joint distribution to be the tensor product of distributions) ,
\[
T^{(m,d)}_{\Strans{[p \boxplus_{d,\lambda} q]}} =_{law} T^{(m,d)}_{\Strans{p}} +T^{(m,d)}_{\Strans{q}},
\]
that is, the random variables have the same distribution.
\end{prop}
\begin{proof}
We have:
\begin{align}
& \big[\mathbb{E}\big[ e^{-T^{(m,d)}_{\Strans{[p \boxplus_{d,\lambda} q]}} \partial_x \partial_y} \big] \{y^{m}x^{d}\} = y^{m-d}[p \boxplus_{d,\lambda} q] (xy)\\
&=y^{m-d}\sum_{k=0}^d (xy)^{d-k}(-1)^k\sum_{i+j=k} p_i q_j \frac{(d-i)!(d-j)!}{d!(d-i-j)!}\frac{(m-i)!(m-j)!}{m!(m-i-j)!}\\
&=  y^{m-d}\sum_{k=0}^d (xy)^{d-k}(-1)^k\sum_{i+j=k} \frac{\mathbb{E}\big[{T^{(m,d)}_{\Strans{p}}}^i \big] m!d! }{i!(m-i)!(d-i)!} \frac{\mathbb{E}\big[{T^{(m,d)}_{\Strans{q}}}^j \big] m!d! }{j!(m-j)!(d-j)!}     \frac{(d-i)!(d-j)!}{d!(d-i-j)!}\frac{(m-i)!(m-j)!}{m!(m-i-j)!}\\
&=  y^{m-d}\sum_{k=0}^d (xy)^{d-k}(-1)^k\sum_{i+j=k} \frac{\mathbb{E}\big[{T^{(m,d)}_{\Strans{p}}}^i \big] \mathbb{E}\big[{T^{(m,d)}_{\Strans{q}}} \big]^j }{i! j!} \frac{m! d!} {(d-i-j)!(m-i-j)!}\\
&=y^{m-d}\sum_{k=0}^d (xy)^{d-k}(-1)^k\sum_{i+j=k}\frac{ \mathbb{E}\big[{(T^{(m,d)}_{\Strans{p}}})^i({T^{(m,d)}_{\Strans{q}}})^j \big] \binom{k}{i} }{k!}\frac{m! d!} {(m-k)!(d-k)!} ~(\textit{by independence})\\
&= \Strans \big[ \sum_{k=0}^d \sum_{i+j=k}\frac{ \mathbb{E}\big[(-T^{(m,d)}_{\Strans{p}})^i(-T_{\Strans{q}})^j \big] \binom{k}{i} }{k!} (\partial_x \partial_y)^{k} \{y^{m}x^{d}\}\big]\\
&= \Strans \big[\sum_{k=0}^d \frac{\mathbb{E}\big[(-1)^k(T^{(m,d)}_{\Strans{p}}+T^{(m,d)}_{\Strans{q}})^k\big] (\partial_x \partial_y)^{k}}{k!} \{y^{m}x^{d}\}\big]\\
&=\mathbb{E}\big[ e^{-(T^{(m,d)}_{\Strans{p}} + T^{(m,d)}_{\Strans{q}}) \partial_x \partial_y} \big] \{y^{m}x^{d}\}.
\end{align}
By Definition~\ref{defT} and unicity from Claim~\ref{unic}, we get:
\[
T^{(m,d)}_{\Strans{[p \boxplus_{d,\lambda} q]}} =_{law} T^{(m,d)}_{\Strans{p}} +T^{(m,d)}_{\Strans{q}}.
\]
\end{proof}
The $T$-transform therefore linearizes the rectangular convolution. We will use this property to define a polynomial of degree $d$ that also shares this linearization property (the rectangular finite $R$-transform.

\subsection{The $R$-transform as a polynomial in partial derivatives}
We start by the definition of the ($m$,$d$)-rectangular finite $R$ transform of a measure $\mu_{\Strans{p_A}}$. First let's introduce a notation.
\begin{defn}[Modulo truncation]
For two formal power series $f(s)=\sum_if_is^i$ and $g(s)=\sum_ig_is^i$ and an integer $d$ we write
\[
f(s)\equiv g(s) \mod[s^{d}]
\]
if
\[
\sum_{i=0}^{d-1}f_is^i= \sum_{i=0}^{d-1}g_is^i.
\]
We will also denote by $f(s)\mod[s^d]$ the polynomial $\sum_{i=0}^{d-1}f_is^i$.
\end{defn}
\begin{rmk} \label{compositionmod}
It was proven in \cite{M} that for any power series $h$, if $f(s)= g(s) \mod[s^{d}]$ then
\[
h[f(s)]\equiv h[g(s)] \mod[s^{d}],
\]
and we also have
\begin{align}
\frac{df}{ds}(s)& \equiv \frac{dg}{ds}(s) \mod[s^{d}] & &  \text{and}&  & sf(s) \equiv sg(s) \mod[s^{d+1}].
\end{align}
\end{rmk}

 \begin{defn}[Rectangular finite $R$-transform] \label{Rtransformcor} We define  $\mathcal{R}^{d,\lambda}_{\Strans{p_A}}(s)$ as the unique polynomial of degree $d$ verifying
 \[
 \mathcal{R}^{d,\lambda}_{\Strans{p_A}}(s) \equiv \frac{-1}{d}s \frac{d}{ds}\log \big( \mathbb{E}\big[ e^{-T^{(m,d)}_{\Strans{p_A}}smd}\big]  \big) \mod[s^{d+1}].
  \]
 \end{defn}
 \begin{rmk}
It will become clear later on why we choose such constants in the definition (it is related to convergence). Notice that  $\mathcal{R}^{d,\lambda}_{\Strans{p_A}}(s)$ has zero constant term so that   $\frac{\mathcal{R}^{d,\lambda}_{\Strans{p_A}}(s)}{s}$ is a polynomial of degree $d-1$. 

\end{rmk}

We can deduce the fundamental additivity property for the rectangular finite $R$-transform:
\begin{thm}[Finite additivity]\label{Rtransform:additivity}
For two polynomials $p$ and $q$ of degree $d$ with nonnegative real roots, we have
\[
\mathcal{R}^{d,\lambda}_{\Strans{[p \recsum q]}}(s) = \mathcal{R}^{d,\lambda}_{\Strans{p}}(s) +\mathcal{R}^{d,\lambda}_{\Strans{q}}(s).
\]
It is the direct analogue of the free probability additivity property that defines the free $R$-rectangular transform:
\begin{align*}
 \Rtransform{\mu_{\Strans{p}} \recsumfree \mu_{\Strans{q}} }(s) &=  \Rtransform{\mu_{\Strans{p}}} (s) + \Rtransform{\mu_{\Strans{q}}} (s) \\
 &=  \Rtransform{\Strans{p}} (s) + \Rtransform{\Strans{q}} (s).
\end{align*}
\end{thm}
\begin{proof}
By definition of the rectangular finite
$R$-transform, we get
\begin{align*}
\mathcal{R}^{d,\lambda}_{\Strans{[p \recsum q]}}(s) & \equiv \frac{-s}{d}\frac{d}{ds} \Big(\log \mathbb{E} \big[e^{-T^{(m,d)}_{\Strans{[p\recsum q]}} smd}\big] \Big)  \mod[s^{d+1}]\\
& \equiv \frac{-s}{d}\frac{d}{ds} \Big( \log \mathbb{E} \big[ e^{(-T^{(m,d)}_{\Strans{p}} - T^{(m,d)}_{\Strans{q}})smd}\big]  \Big)  \mod[s^{d+1}]  \text{   using Corollary \ref{corT}}\\
&\equiv \frac{-s}{d}\frac{d}{ds} \Big( \log \mathbb{E} \big[ e^{-T^{(m,d)}_{\Strans{p}} smd}\big]  + \log \mathbb{E} \big[ e^{-T^{(m,d)}_{\Strans{q}} smd}\big]\Big)  \mod[s^{d+1}]  \\
&= \mathcal{R}^{d,\lambda}_{\Strans{p}}(s) +\mathcal{R}^{d,\lambda}_{\Strans{q}}(s)  \mod[s^{d+1}] \\
&= \mathcal{R}^{d,\lambda}_{\Strans{p}}(s) +\mathcal{R}^{d,\lambda}_{\Strans{q}}(s)  
\end{align*}
as $\mathcal{R}^{d,\lambda}_{\Strans{q}}(s) $, $\mathcal{R}^{d,\lambda}_{\Strans{p}}(s)$ and $\mathcal{R}^{d,\lambda}_{\Strans{[p \recsum q]}}(s) $ are all polynomials of degree at most $d$. 
\end{proof}
We can also deduce the following meaningful representation:
\begin{prop}[Inversion formula]\label{inversionf}
For any monic polynomial $p$ with nonnegative roots of degree $d$, 
\[
 y^{m-d}p(xy)= e^{- [\widehat{\mathcal{R}}^{d,\lambda}_{\Strans{p}}(\partial_x \partial_y) ]} \{y^{m}x^{d}\}
 \]
 where $\widehat{\mathcal{R}}^{d,\lambda}_{\Strans{p}}(s):= d\int_s \frac{ \mathcal{R}^{d,\lambda}_{\Strans{p}}(\frac{s}{md})}{(\frac{s}{md})}$ (the primitive which is zero at zero). It follows that $\mathcal{R}^{d,\lambda}_{\Strans{p}}$ characterizes uniquely $p$ in the sense that we can recover $p$ if we know $\mathcal{R}^{d,\lambda}_{\Strans{p}}$. It should be noted that as the transform is a function of the roots only, we need to require $p$ to be monic to be able to recover the polynomial.
\end{prop}
\begin{proof}
Consider the polynomial $Q_p$ of degree $d$ such that
\begin{equation}\label{eqn:diffequation}
Q_p [\frac{{\partial_x\partial_y}}{md}]\{y^{m}x^{d}\} =  y^{m-d}p(xy)=   \mathbb{E}\big[ e^{-T^{(m,d)}_{\mu_{\Strans{p}}} \partial_x \partial_y} \big] \{y^{m}x^{d}\}.
\end{equation}
Its coefficients can be found explicitly as the reverse of the coefficients of $p$ renormalized. Furthermore it is clear that such a polynomial is defined uniquely. 
In particular, we have that $Q_p$ is the following truncated polynomial:
\begin{equation}\label{eqn:sequation}
Q_p[\frac{s}{md}] \equiv  \mathbb{E}\big[ e^{-T^{(m,d)}_{\mu_{\Strans{p}}}s} \big]  \mod[s^{d+1}]\\.
\end{equation}
Indeed, if a polynomial $Q$ verifies \eqref{eqn:sequation} (there is only one such possible polynomial of degree $d$), then it verifies \eqref{eqn:diffequation} by plugging $s=\partial_x\partial_y$ into  \eqref{eqn:sequation}. And we have necessarily $Q=Q_p$. We then get using Remark~\ref{compositionmod} and the fact that $\log(1+s)$ expands as a power series:
\begin{align*}
-\log [Q_p(s)] &\equiv  -\log \big[\mathbb{E}\big( e^{-T^{(m,d)}_{\mu_{\Strans{p}}}smd} \big) \big] \mod[s^{d+1}]\\
 - s \frac{d \log [Q_p(s)]}{ds}&\equiv -s \frac{d}{ds}\log \big[\mathbb{E}\big( e^{-T^{(m,d)}_{\mu_{\Strans{p}}}smd} \big) \big] \mod[s^{d+1}]\\
 - s \frac{d \log [Q_p(s)]}{ds}&\equiv d \mathcal{R}^{d,\lambda}_{\Strans{p}}(s) \mod[s^{d+1}]\\
\log [Q_p(s)] &\equiv- d \int_s \frac{\mathcal{R}^{d,\lambda}_{\Strans{p}}(s)}{s}\mod[s^{d+1}] \\ 
 Q_p(s)& \equiv e^{-d \int_s \frac{\mathcal{R}^{d,\lambda}_{\Strans{p}}(s)}{s}} \mod[s^{d+1}].
\end{align*}
We plug in $s=\partial_x\partial_y$ so that
\begin{align*}
  y^{m-d}p(xy)= Q_p(\frac{\partial_x\partial_y}{md}) \{y^mx^d \}&=  \Big(e^{- \widehat{\mathcal{R}}^{d,\lambda}_{\Strans{p}}(\partial_x\partial_y) } \mod[(\partial_x\partial_y)^{d+1}] \Big)\{y^mx^d \} \\
  &= e^{- \widehat{\mathcal{R}}^{d,\lambda}_{\Strans{p}}(\partial_x\partial_y) }\{y^mx^d \} 
\end{align*}
as $ h(\partial_x\partial_y)\{y^mx^d \} = \big[h(\partial_x\partial_y) \mod[(\partial_x\partial_y)^{d+1}] \big]\{y^mx^d \} $, because  $(\partial_x\partial_y)^{k}\{y^mx^d \}=0$ for $k \geq d+1$.
\end{proof}

\section{Adapting free probability notions to polynomials}\label{adapt}
\subsection{Redefining transforms for compact discrete symmetric measures}

\begin{defn}\label{Htransform:defn}
The $\lambda$-rectangular Cauchy transform of a polynomial $p$ with all roots nonnegative or equivalently of the symmetric polynomial $\Strans{p}$  is given by
\[
\Hcauchy{\Strans{p}}(x):= \mathcal{G}_{\mathbb{S}p}(x) \Big(  \lambda \mathcal{G}_{\mathbb{S}p}(x)+ (1-\lambda) {\mathcal{G}}_0(x)  \Big)  \textit{\ \ \ for \ } x> \sqrt{\maxroot{p}}.
\]
\end{defn}
Recall that ${\mathcal{G}}_0(x) = \frac{1}{x}$. 
It is easy to check that
\[
\Hcauchy{\Strans{p}}(x)=H^{\lambda}_{\mu_{\Strans{p}}} (\frac{1}{x}).
\]
This definition is a slightly modified version of the Benaych Georges' transform fit for polynomials. Our definition also incorporates the symmetrization. 
We can rewrite it more explicitly if $p$ is of degree $d$ as:
\[
\Hcauchy{\Strans{p}}(x)= \frac{x}{d}\frac{p'(x^2)}{p(x^2)} \Big( 
\lambda\frac{1}{d}\frac{x p'(x^2)}{p(x^2)}+ \frac{(1-\lambda)}{x}\Big).
\]

Note that we get the expansion around  $x= \infty$ :
\[
\Hcauchy{\Strans{p}}(x)= \sum_{i=1}^{\infty} h_i^{(p)} \frac{1}{x^{2i}}.
\]
where  $h_1^{(p)}=1$. Also, after the last root, the $H$-transform will be monotonous decreasing. It leads to the following.
\begin{lemma} \label{Hinverse:defn} $\Hcauchy{\Strans{p}}$ is a bijection from $[\sqrt{\maxroot{p}}, +\infty]$ to $ [0,+\infty]$. We denote the inverse by $\Jcauchy{\Strans{p}}$. In particular, for $x> 0$ 
\[
\Hcauchy{\Strans{p}} \circ \Jcauchy{\Strans{p}} [x]=x,
\]
and for  $ x>\sqrt{\maxroot{p}}$
\[
\Jcauchy{\Strans{p}} \circ \Hcauchy{\Strans{p}} [x]=x.
\]
\end{lemma}

\begin{lemma}\label{Hinverse:defn}
We can get an expression for $\Jcauchy{\Strans{p}}$ in terms of the free probability power series inverse $[H^{\lambda}_{\mu_{\Strans{p}}} ]^{-1}$. Given the quadratic nature in terms of the inverse variable, it makes sense for symmetry to replace $x$ by $s^2$, which is relevant as all the quantities we consider are positive. It will avoid us to restrict to positive variables. For s in a neighbohood of $0$,
\[
\big[ \Jcauchy{\Strans{p}}(s^2)\big]^2= \frac{1}{[H^{\lambda}_{\mu_{\Strans{p}}} ]^{-1}(s^2)}.
\]
\end{lemma}
\begin{proof}
Use the formula $\Hcauchy{\Strans{p}}(y)=H^{\lambda}_{\mu_{\Strans{p}}} (\frac{1}{y^2})$.
\end{proof}

\begin{cor}
$s^2\big[ \Jcauchy{\Strans{p}}(s^2)\big]^2$ is analytic for $s$ in a neighborhood of $0$, even though  $\Jcauchy{\Strans{p}}(s^2)$ is not. It is important to emphasize that  $\Jcauchy{\Strans{p}}(s^2)$ is not a power series in $s^2$.  
\end{cor}
\begin{proof}
Let us write $[H^{\lambda}_{\mu_{\Strans{p}}} ]^{-1}(s)= s+ \sum_{i=2}^{\infty}\beta_is^{i}$ (which can be deduced from Definition~\ref{defnHtransform}) for some coefficients $\beta_i$. 
\[
s^2\big[ \Jcauchy{\Strans{p}}(s^2)\big]^2= \frac{s^2}{s^2+ \sum_{i=2}^{\infty}\beta_is^{2i}}=\frac {1}{1+\sum_{i=1}^{\infty}\beta_{i+1}s^{2i}},
\]
which is analytic (multiplying by $s^2$ gets rid of the singularity at $0$).
\end{proof}

We are now able to get a more simple formula for the Benaych-Georges' rectangular $R$-transform applied to a symmetric discrete measure. 

\begin{lemma}\label{Rtransformform}
\[
\Rtransform{\mu_{\Strans{p}}}(s^2)= \frac{-\lambda -1}{2\lambda} +\sqrt{ \frac{(\lambda-1)^2}{4\lambda^2}+  \frac{ s^2 [\Jcauchy{\Strans{p}}(s^2)]^2 }{\lambda} }
\]

\end{lemma}

\begin{proof}
Using Definition~\ref{Rrectdefn}, we get:
\begin{align*} \Rtransform{\mu_{\Strans{p}}}(s^2)= U \Big ( \frac{s^2}{[{H_{\mu_{\Strans{p}}}^\lambda}]^{-1}(s^2)} -1\Big)  &= \frac{-\lambda -1}{2\lambda} +\frac{\sqrt{(\lambda+1)^2+4\lambda \big(\frac{s^2}{{[H^{\lambda}_{\mu_{\Strans{p}}} ]^{-1}(s^2) }} -1\big)} } {2\lambda} \\
&= \frac{-\lambda -1}{2\lambda} +\frac{\sqrt{(\lambda-1)^2+ \frac{4\lambda s^2}{{[H^{\lambda}_{\mu_{\Strans{p}}} ]^{-1}(s^2) }} } } {2\lambda}\\
&=  \frac{-\lambda -1}{2\lambda} +\sqrt{ \frac{(\lambda-1)^2}{4\lambda^2}+  \frac{ s^2 [\Jcauchy{\Strans{p}}(s^2)]^2 }{\lambda} }.
\end{align*}
\end{proof}

\begin{defn}
For $p$ with all roots nonnegative, we define:
\[
\Rtransform{\Strans{p}}(s^2) := \Rtransform{\mu_{\Strans{p}}}(s^2).
\]
It is important to stress that here $\Rtransform{\Strans{p}}(s^2)$  is indeed a power series in $s^2$ and not just a function.
\end{defn}

\subsection{Truncating bivariate power series} \label{truncation}

We start by generalizing the notion of truncation to bivariate power series. For  $f(r,t)= \sum_{i,j} a_{i,j} r^it^j$, we denote by
\[
f \mod [ <r,t>^d]:= \sum_{i+j \leq d-1} a_{i,j}r^it^j .
\]
We will write:
\[
f\equiv g \mod[ <r,t>^d]
\]
when $f-g \mod [<r,t>^d] = 0 $. The following properties follow easily. If
\begin{align*}  a(r,t) &\equiv b(r,t) \mod[<r,t>^d]        &      c(r,t) &\equiv d(r,t) \mod[<r,t>^d]
\end{align*}
then, there is additivity and multiplicativity:
\begin{align*}
 a(r,t) +c(r,t) &\equiv b(r,t)+d(r,t) \mod[<r,t>^d]        &      a(r,t)c(r,t) &\equiv b(r,t)d(r,t) \mod[<r,t>^d].
\end{align*}
Therefore for any power series $ h$,
\[
h [ a(r,t) ] \equiv h [ b(r,t) ] \mod [<r,t>^d].
\]
We can also differentiate the modulo equalities, up to losing one degree:
\begin{align*}  \frac {\partial}{\partial r}  a(r,t) & \equiv \frac {\partial}{\partial r}  b (r,t) \mod[<r,t>^{d-1}]     &     \frac {\partial}{\partial t}  a(r,t) &\equiv \frac {\partial}{\partial t}  b (r,t).\mod[<r,t>^{d-1}]\end{align*}

\subsection{$L_q$-norms on intervals depending on the largest root}
\begin{defn}
Take a two-variable continuous  function $f(x,y)$ that  is positive on a bidimensional interval $X$. Define the $L_q$-norm associated with this interval as
\[
\| f(x,y) \|_q:=\Big( \int \int_{(x,y) \in X} f(x,y)^q  dx dy\Big)^{\frac{1}{q}}.
\]
\end{defn}

\begin{defn}
We will denote, similarly, by
\[
\| f(x,y) \|_{\infty}:= \sup_{(x, y) \in X} | f(x,y) |.
\]
\end{defn}
\begin{lemma} \label{infinityconvergence}
If $f$ is such that its $L_q$ and $L_{\infty}$ norms are finite, then
\[
\lim_{q \to \infty}\| f(x,y) \|_q= \| f(x,y) \|_{\infty}.
\]
\end{lemma}
In all the following, the function $f(x,y)$ that we will consider will contain some problematic polynomial factor that doesn't have constant sign: $p_A(xy)$. We will therefore choose an interval X that depends on the polynomial (or the operator), more specifically if we denote by $R$ the largest root of $p$, we take
\[
X := [\sqrt{\maxroot p}, +\infty] \times [\sqrt{\maxroot p }, +\infty] =  [\sqrt{R}, +\infty] \times [\sqrt{R}, +\infty] .
\]
This ensures that for $(x,y)$ in this interval, the product $xy$ is above the largest root and the polynomial is positive.

\subsection{Rectangular Fuglede-Kadison determinant and potential}
We first recall the usual simple definition of the Fuglede-Kadison determinant. For a finite dimensional positive definite $d \times d $ matrix, it corresponds to the following normalized determinant:

\[
 \Delta^{+}(A) := \det(A)^{1/d}.
\]

But the main object needed will be the bivariate following object, defined for 
a general rectangular matrix $A$.
\begin{defn}
For  $\lambda \in ]0,1]$, and positive integers $m$, $d$ such that $\frac{d}{m}=\lambda$ and for $(x,y) \in X$, introduce the bivariate $\lambda$-rectangular Fuglede-Kadison determinant of an $m \times d$ matrix $A$:
\[
\Delta_{\lambda}^{+}(x,y,A):= y^{\frac{1-\lambda}{\sqrt{\lambda}}} \big[\Delta^{+}\big(xyI_d -A^{T}A\big)\big]^{\sqrt{\lambda}}.
\]
 It is  well defined because for  $(x,y) \in X$, the polynomial inside takes positive values. This definition shows that it doesn't depend on the dimensions, only on the root distribution and the parameter $\lambda$: indeed changing $d$ will not affect this quantity if the distribution is kept identical (see Remark~\ref{rmkdim} just below)   . We can however expand and we get
\[
\Delta_{\lambda}^{+}(x,y,A)= y^{\frac{m-d}{\sqrt{md}}}[p_A(xy)]^{\frac{1}{\sqrt{md}}}
\]
for $(x,y) \in X$.  

\begin{rmk}\label{rmkdim}
We normalize the characteristic polynomial so that the polynomials associated to two matrices of different sizes but with the same underlying measure on the roots don't differ. Take $A_1$, $ m_1 \times d_1$ and $A_2$, $ m_2 \times d_2$, two rectangular matrices such that $\frac{d_1}{m_1}=\frac{d_2}{m_2}=\lambda$ and $\mu_{p_{A_1}} = \mu_{p_{A_2}}$. Then
\[
\Delta_{\lambda}^{+}(x,y,A_1)= \Delta_{\lambda}^{+}(x,y,A_2).
\]
Therefore we can increase the dimensions and the normalized polynomials stay the same. 
In all the following we will denote by
\[
\Delta_{\lambda}^{+}(x,y,\mu_A):=\Delta_{\lambda}^{+}(x,y,A)
\]
if  $\mu_{\Strans{p_{A}}}= \mu_A$. 
\end{rmk}
\end{defn}

\begin{defn}[Fuglede-Kadison Potential]
For a given symmetric compact discrete measure $\mu_A$, we consider the following $\lambda$-Fuglede Kadison potential defined on $X$:
\[
F^{\lambda}_{\mu_A}(x,y):= - \log \Delta_{\lambda}^{+}(x,y,\mu_A).
\]
Rewrite $F^{\lambda}_{\mu_A}$ as 
\[
F^{\lambda}_{\mu_A}(x,y)= - \frac{1}{\sqrt{md}}\log [y^{m-d}{p_A}(xy)].
\]
\end{defn}

\begin{lemma}
$F^{\lambda}_{\mu_A}$ is convex and $\partial_x {F^{\lambda}_{\mu_A}} (z,z)\partial_y {F^{\lambda}_{\mu_A}}(z,z)=\Hcauchy{\mu_A}(z) $ for $z>0$.
\end{lemma}
\begin{proof}
\begin{align*}
\partial_x {F^{\lambda}_{\mu_A}}(z,z)\partial_y {F^{\lambda}_{\mu_A}}(z,z)&= \frac{1}{md}  \Big( \frac{zp_A'(z^2)}{p_A(z^2)} \Big) \Big( \frac{zp_A'(z^2)}{p_A(z^2)} +\frac{(m-d)}{z}  \Big) =   \Big(\frac{1}{d}  \frac{zp_A'(z^2)}{p_A(z^2)} \Big) \Big( \frac{d}{m}\frac{1}{d}\frac{zp_A'(z^2)}{p_A(z^2)} +\frac{(m-d)}{m}\frac{1}{z}  \Big) \\
&=\Hcauchy{\Strans{p_A}}(z).
\end{align*}
As for the convexity it suffices to compute the second derivatives and show that $\partial_{xx} {F^{\lambda}_{\mu_A}} \geq 0$, $\partial_{yy} {F^{\lambda}_{\mu_A}} \geq 0$ and $ \partial_{xx} {F^{\lambda}_{\mu_A}}\partial_{yy} {F^{\lambda}_{\mu_A}}- (\partial_{xy} {F^{\lambda}_{\mu_A}})^2 \geq 0$. But we have

\begin{align*}
\partial_{xx} {F^{\lambda}_{\mu_A}}(x,y)&= \frac{1}{\sqrt{md}}\frac{y^2}{p_A^2(xy)} \Big( p_A'^2(xy)- p_A''(xy)p_A(xy)\Big) \\
\partial_{yy} {F^{\lambda}_{\mu_A}}(x,y)&= \frac{1}{\sqrt{md}} \frac{x^2}{p_A^2(xy)} \Big( p_A'^2(xy)- p_A''(xy)p_A(xy)\Big) + \frac{m-d}{\sqrt{md}}\frac{1}{y^2}\\
\partial_{xy} {F^{\lambda}_{\mu_A}}(x,y)&= \frac{1}{\sqrt{md}}\frac{xy}{p_A^2(xy)} \Big(p_A'^2(xy)- p_A''(xy)p_A(xy) \Big).
\end{align*}
We notice that as $p_A$ is real-rooted,
\[
 p_A'^2(xy)- p_A''(xy)p_A(xy) \geq 0
 \]
 by Laguerre's inequality, and therefore we have $\partial_{xx} {F^{\lambda}_{\mu_A}} \geq 0$, $\partial_{yy} {F^{\lambda}_{\mu_A}} \geq 0$. Lastly,
\[
 \partial_{xx} {F^{\lambda}_{\mu_A}}\partial_{yy} {F^{\lambda}_{\mu_A}}- (\partial_{xy} {F^{\lambda}_{\mu_A}})^2 =  \frac{m-d}{md} \frac{1}{{p_A}^2(xy)} \Big( p_A'^2(xy)- p_A''(xy)p_A(xy)\Big) \geq 0.
 \]
 \end{proof}

\subsection{Bidimensional Legendre transform}

The main tool to inverse gradients of convex functions is the Legendre transform.
We first recall the definition the bidimensional Legendre transform and basic porperties. 
Denote by $X^{*}= \nabla f(X)$,  a convex subset of $\mathbb{R}^2$, the image 
of $X$ by the gradient of $f$. 
\begin{defn} [bidimensional Legendre transform]
If $f(x,y)$ is a  convex function on  $X$  we define
for $(r,t) \in X^{*}$ :
\begin{align*}
f^{*}(r,t)&:=  \|  xr+ yt - f(x,y)\|_{\infty}\\
&= \sup_{(x,y)\in X}\{ xr+ yt - f(x,y)\} 
\end{align*}
\end{defn}
\begin{lemma} \label{gradientinv}
If $f$ is strictly convex on $X$, then we have for all $(x,y) \in X, (r,t)  \in 
X^*$
\[
[\nabla f^{*} \circ \nabla f] (x,y)= (x,y),
\]

\[
[\nabla f \circ \nabla f^*] (r,t)= (r,t).
\]

\end{lemma}

\begin{lemma}
For all $(r,t)\in X^*$,
\[
f^{*}(r,t)= \log \| e^{xr+yt-f(x,y)}\|_{\infty}.
\]
\end{lemma}
\begin{proof}
It follows from the fact that $\|   e^{g(x,y)}\|_{\infty}= e^{\|   g(x,y)\|_{\infty}}$.
\end{proof}

\subsection{Approximation of the inverse transform}
The goal in this section is to derive an explicit expression for $\Jcauchy{\Strans{p_A}}(s^2)$ using the bivariate tools introduced in the last section.

\begin{lemma}[Inverting the rectangular Cauchy transform] \label{Hconvergence} 
For $s \neq 0$, then we can exhibit $r_{p_A}(s)$ and $t_{p_A}(s)$ functions of $s$, such that  
$r_{p_A}(s)t_{p_A}(s)=s^2$,  and 
\[
  {\Jcauchy{\mu_A}}(s^2)^2 = {\Jcauchy{\Strans{p_A}}}(s^2)^2 =   {\frac{\partial}{\partial r} \log \|e^{-xr-yt}\Delta_{\lambda}^{+}(x,y,\mu_A)  
\|_{\infty}  \frac{\partial}{\partial t} \log \|e^{-xr-yt} 
\Delta_{\lambda}^{+}(x,y,\mu_A) \|_{\infty}}_{\big|_{r=r_{p_A}(s), t=t_{p_A}(s)}}.
\]
 
\end{lemma}
\begin{proof}
Recall that
\[
F^{\lambda}_{\mu_A}(x,y)= - \log \Delta_{\lambda}^{+}(x,y,\mu_A).
\]
Take  $s \neq 0$.We have by definition  $\Hcauchy{\Strans{p_A}} \circ \Jcauchy{\Strans{p_A}} (s^2) = s^2$. Call $z_{p_A}(s):= \Jcauchy{\Strans{p_A}} (s^2)$ to alleviate the computations. Ideally, we would want some explicit power series of $z_{p_A}$ in $s$. This being too complicated, we look for an expression that would explicitly depend on  $r_{p_A}$ and $t_{p_A}$ such that $r_{p_A}t_{p_A}=s^2= \Hcauchy{\Strans{p_A}}(z_{p_A})$.

Define: 
\begin{align*}
r_{p_A}(s):=&- \partial_xF^{\lambda}_{\mu_A}(z_{p_A},z_{p_A}) = \frac{1}{\sqrt{md}} \frac{z_{p_A}p_A'(z_{p_A}^2)}{p_A(z_{p_A}^2)},\\
t_{p_A}(s):=&- \partial_yF^{\lambda}_{\mu_A}(z_{p_A},z_{p_A})= \frac{1}{\sqrt{md}} \frac{z_{p_A}p_A'(z_{p_A}^2)}{p_A(z_{p_A}^2)} + \frac{1}{\sqrt{md}}\frac{m-d}{z_{p_A}}.
\end{align*}

Given that $ z_{p_A}> \sqrt{R}$ as mentioned, we get that $r_{p_A}>0$, $t_{p_A}>0$. 
Furthermore, using the fact that the gradient of the Legendre transform is the 
inverse of the gradient of $F$, like stated in Lemma~\ref{gradientinv}, we obtain,
applying the equality at $(z_{p_A},z_{p_A})$
\[
 \begin{bmatrix}  \partial_r ({F^{\lambda}_{\mu_A}})^{*} [ \partial_xF^{\lambda}_{\mu_A}(z_{p_A},z_{p_A}), \partial_y F^{\lambda}_{\mu_A}(z_{p_A},z_{p_A})   ] \\    \partial_t ({F^{\lambda}_{\mu_A}})^{*} [ \partial_xF^{\lambda}_{\mu_A}(z_{p_A},z_{p_A}), \partial_y F^{\lambda}_{\mu_A}(z_{p_A},z_{p_A})   ]   \end{bmatrix} = \begin{bmatrix} z_{p_A} \\ z_{p_A} \end{bmatrix}.
\]
Making the product of the two lines we get
\[
\partial_r ({F^{\lambda}_{\mu_A}})^{*} (-r_{p_A},-t_{p_A})  \partial_t ({F^{\lambda}_{\mu_A}})^{*} (-r_{p_A},-t_{p_A})= z_{p_A}^2.
\]
Now let's rewrite the left hand side quantity and conclude:
\begin{align}
&\partial_r ({F^{\lambda}_{\mu_A}})^{*} (-r_{p_A},-t_{p_A})  \partial_t ({F^{\lambda}_{\mu_A}})^{*} (-r_{p_A},-t_{p_A})= \frac{\partial}{\partial 
r} \log \|e^{-xr-yt- F^{\lambda}_{\mu_A}(x,y)} \|_{\infty}  \frac{\partial}{\partial t} \log 
\|e^{-xr-yt- F^{\lambda}_{\mu_A}(x,y)} \|_{\infty}\big|_{r=r_{p_A}, t=t_{p_A}} \\
&=  \frac{\partial}{\partial r} \log \|e^{-xr-yt}\Delta_{\lambda}^{+}(x,y,\mu_A)  
\|_{\infty}  \frac{\partial}{\partial t} \log \|e^{-xr-yt} 
\Delta_{\lambda}^{+}(x,y,\mu_A) \|_{\infty}\big|_{r=r_{p_A}, t=t_{p_A}}.
\end{align}
\end{proof}

\begin{cor}
$r_{p_A}(s)$ and $t_{p_A}(s)$ are analytic functions of $s$ in the neighborhood of zero.
\end{cor}

\begin{proof}
The proofs are similar. We start by the decomposition:
\[
r_{p_A}(s) = \sum_{i=1}^{\infty} \rho_i \frac{1}{z_{p_A}^i}
\]
where $z_{p_A}= \Jcauchy{\mu_A}(s^2)$. Now remember that
\[
\frac{1}{z_{p_A}}= \sqrt {[H^{\lambda}_{\mu_A} ]^{-1}(s^2)} = \sum_{i=1}^{\infty} z_i s^i.
\]
The result follows by composition of power series around the origin. 
\end{proof}

\begin{rmk}
Notice that in the square case, $r_{p_A}=s$ and $t_{p_A}=s$. In the rectangular case, they are approximations of the identity, and in all the following we will show that in this case, the higher order terms in the power expansions in $s$ can still be neglected (but it is not trivial).
\end{rmk}
Computing the infinity norm is not easy, and we try to reach it through $L_q$-norms and taking limits as $q$ goes to infinity. 
 
 \begin{defn}
  For an $m \times d$ matrix A with singular distribution $\mu_A$, define the finite 
  inverse rectangular transform as
 \[
\mathcal{Q}_{\mu_A}^{d,\lambda}(s):= {\frac{\partial}{\partial r}\big( \log 
 \|e^{-xr-yt} \Delta_{\lambda}^{+}(x,y,\mu_A)\|_{\sqrt{md}} \big)  
 \frac{\partial}{\partial t} \big(\log\| e^{-xr-yt} 
 \Delta_{\lambda}^{+}(x,y,\mu_A)\|_{\sqrt{md}} \big)}_{\big|_{r=r(s), t=t(s)}}.
 \]
 
 \end{defn}
 \begin{lemma} \label{lemmalpinfinity}
 For $s$ small enough such that $ {\Jcauchy{\mu_A}}(s^2)$ is well defined and finite,
 \[
  \lim_{\substack{d\to \infty \\}}  \mathcal{Q}_{\mu_A}^{d,\lambda}(s) =  {\Jcauchy{\mu_A}}(s^2)^2.
  \]
  The limit depends only on $\mu_A$ (which is independent from the dimension).
 \end{lemma}
 
 \begin{rmk}
 As  $s^2{\Jcauchy{\mu_A}}(s^2)^2 $ is analytic in $s$, we will show in the next section that similarly $s^2\mathcal{Q}_{\mu_A}^{d,\lambda}(s)$ is analytic too.
 \end{rmk}
 We will need to use the following results of uniform convergence to prove Lemma~\ref{lemmalpinfinity}.
 \begin{lemma} [Second theorem of Dini] \label{secondDini}
 Assume that are given a sequence of functions $\big(f_n(x) \big)_n $defined on a segment of the real line and such that
 \begin{itemize}
 \item f is continuous,
 \item there is monotonicity in $x$: for $x \leq y, f_n(x) \leq f_n(y)$,
 \item  there is pointwise convergence: $f_n(x)$ converges to $f(x)$,
 \end{itemize}
 then the convergence is uniform.
 \end{lemma}
 \begin{lemma}[Derivative uniform convergence]  \label{derivativeuniform}
 If $f_n$ converges pointwise to $f, f_n' $ converges uniformly to $f'$, then $f_n$ converges uniformly to $f$and $f'=\lim_{n \to \infty} f_n'$.
 \end{lemma}
 We will also need to use convexity of Laplace transforms.
 \begin{lemma} [Convexity of Laplace transform] \label{Laplaceconvex} 
 Consider a function $g$ that is positive and sufficiently smooth so that it's Laplace tranform is well defined on a domain $X$,
 \[
 L(s):= \log \big(\int_{X}e^{-su}g(u) du \big).
 \]
 Then $L$ is convex, therefore $L'$ is monotonous.
 \end{lemma}
 \begin{proof}
 \[
 L''(s)= \frac{(\int_{X} e^{-su}g(u) du )(\int_{X} e^{-su}g(u) u^2 du) -(\int_{X} e^{-su}g(u) u du)^2 }{\big(\int_{X} e^{-su}g(u) du \big)^2},
 \]
 and the numerator is positive using the inequality of Cauchy Schwarz.
 \end{proof}
 \begin{proof} [Proof of Lemma \ref{lemmalpinfinity}]
 \[
  \lim_{\substack{d\to \infty \\ d= m\lambda}}  \mathcal{Q}_{\mu_A}^{d,\lambda}(s)  = \lim_{\substack{d\to \infty \\ d= m\lambda}} \frac{\partial}{\partial 
r} \log \|e^{-xr-yt}\Delta_{\lambda}^{+}(x,y,\mu_A)  \|_{\sqrt{md}}  
\frac{\partial}{\partial t} \log \|e^{-xr-yt} \Delta_{\lambda}^{+}(x,y,\mu_A) 
\|_{\sqrt{md}}.
 \]
 The goal is to show that we can permute the limit and the derivative signs. We would get the result as
 \begin{align}
  &\lim_{\substack{d\to \infty \\ d= m\lambda}} \frac{\partial}{\partial 
r} \log \|e^{-xr-yt}\Delta_{\lambda}^{+}(x,y,\mu_A)  \|_{\sqrt{md}}  
\frac{\partial}{\partial t} \log \|e^{-xr-yt} \Delta_{\lambda}^{+}(x,y,\mu_A) 
\|_{\sqrt{md}} \\
&=  \frac{\partial}{\partial r} \log \|e^{-xr-yt}\Delta_{\lambda}^{+}(x,y,\mu_A)  
\|_{\infty}  \frac{\partial}{\partial t} \log \|e^{-xr-yt} 
\Delta_{\lambda}^{+}(x,y,\mu_A) \|_{\infty}
 \end{align}
 by Lemma \ref{infinityconvergence}. 
Now let's do only the first term (derivative in $r$), the other one is done the same way.  Define:
\[
f_d(r):= \log \|e^{-xr-yt}\Delta_{\lambda}^{+}(x,y,\mu_A)  
\|_{\sqrt{md}}.
\]
Now we can use Lemma \ref{Laplaceconvex} to get that $f_d'(r)$ is monotonous in $r$ therefore we can apply Lemma \ref{secondDini} to get uniform convergence of the derivatives and eventually we can apply Lemma \ref{derivativeuniform} to conclude that  $ \big[\lim_{d \to \infty} f_d(r) \big]' = \lim_{d \to \infty} f_d'(r)$.

\end{proof}

 \section{The modified rectangular finite $R$-transform and convergence}\label{modifiedtransform}
We define in this section a modified finite  $R$-transform as  a quantity that converges by design to the analytical rectangular free probability $R$-transform defined by Benaych Georges. And we then operate the shift from the convergence of the modified rectangular finite $R$-transform to the actual rectangular finite $R$-transform.
\subsection{The modified finite $R$-transform definition and convergence}

Denote by
\[
\mathcal{K}_{\mu_A}^{d,\lambda} (s):= \sqrt{\frac{1}{\lambda} s^2 \mathcal{Q}_{\mu_A}^{d,\lambda}(s) + \frac{(1-\lambda)^2}{4\lambda^2}   }.
\]
Using the expression of $\mathcal{Q}_{\mu_A}^{d,\lambda}(s)$ computed in the next section, we easily get
\[
 \mathcal{K}_{\mu_0}^{d,\lambda}  =  \Big(\frac{(\lambda+1)}{2\lambda}+\frac{1}{d}\Big),
\]
where $\mu_0$ is a Dirac mass at zero. 
\begin{defn}[Modified rectangular finite $R$-transform]
 For all $s>0$, the modified rectangular finite $R$-transform is defined as
\[
\widetilde{\mathcal{R}}^{d,\lambda}_{\mu_A}(s):= \mathcal{K}_{\mu_A}^{d,\lambda} (s)- \mathcal{K}_{\mu_0}^{d,\lambda}.
\]
\end{defn}
\begin{rmk}
It will follow from \ref{analyQ}  below that $\widetilde{\mathcal{R}}^{d,\lambda}_{\mu_A}(s)$ can be expressed as a power series in $s$ like  $\Rtransform{\mu_A}(s^2)$ (which makes the definition not totally arbitrary). Also note that by definition $\widetilde{\mathcal{R}}^{d,\lambda}_{\mu_A}(s)$ has no constant term. 
\end{rmk}
We can explicitly see the convergence to the free probability transform as
\[
\widetilde{\mathcal{R}}^{d,\lambda}_{\mu_A}(s)=  \sqrt{\frac{1}{\lambda} s^2 \mathcal{Q}_{\mu_A}^{d,\lambda}(s) + \frac{(1-\lambda)^2}{4\lambda^2}   }  -  \Big(\frac{(\lambda+1)}{2\lambda}+\frac{1}{d}\Big)
\]
and
\[
\Rtransform{\mu_A}(s^2)= \sqrt{ \frac{ s^2 [\Jcauchy{\mu_A}(s^2)]^2 }{\lambda} + \frac{(\lambda-1)^2}{4\lambda^2} } - \frac{\lambda +1}{2\lambda}.
\]
As on the other hand, using the convergence Lemma \ref{Hconvergence}, for s small enough such that ${\Jcauchy{\mu_A}}(s^2)$ is well defined:
\[
 \mathcal{Q}_{\mu_A}^{d,\lambda}(s)  \xrightarrow{d\to \infty}    {\Jcauchy{\mu_A}}(s^2)^2,
\]
we get:
\begin{thm}[Convergence of the modified finite $R$-transform to the free $R$-transform] \label{Convmodified}
For $s$ small enough, such that   $\Rtransform{\mu_A} (s^2)$ is absolutely convergent,
\[
\widetilde{\mathcal{R}}^{d,\lambda}_{\mu_A}(s)  \xrightarrow{d\to \infty}    \Rtransform{\mu_A} (s^2).
\]
\end{thm}
\begin{rmk}
$\widetilde{\mathcal{R}}^{d,\lambda}_{\mu_A}(s) $ is a power series in $s$ but in the limit ${\substack{d\to \infty }}$ becomes a power series in $s^2$ like $\Rtransform{\mu_A} (s^2)$.
\end{rmk}

\subsection{Relating the modified finite $R$-transform to the  finite $R$-transform }
The goal is now to get some explicit expression for $\widetilde{\mathcal{R}}^{d,\lambda}_{\mu_A}(s) $ , or at least the first coefficients in the power series expansion, so that we can relate it to the actual $R$-trasnform $\mathcal{R}^{d,\lambda}_{\mu_A}(s)$. 
\subsubsection{From integrals to power series}

In this section, we explicitly integrate and group summation terms. 
\begin{lemma} \label{integral} For $r,t>0$, we get
\[
\|e^{-xr-yt} \Delta_{\lambda}^{+}(x,y,\mu_A)\|_{m\sqrt{\lambda}}^{m\sqrt{\lambda}}=
 \frac{m! d!}{(rm \sqrt{\lambda R})^{d+1} (tm\sqrt{\lambda R})^{m+1}}\sum_{i=0}^d \frac{E[(-T^{(m,d)}_{\Strans{p_A}}rtmd)^i]}{i!}L_i^{d,\lambda}(r) M_i^{d,\lambda}(t)
\]
  where
  \begin{align*}
   &L_i^{d,\lambda}(r):=e^{-rm\sqrt{\lambda R}} \sum_{l=0}^{d-i} \frac{(rm\sqrt{\lambda R})^l}{l!}  &  M_i^{d,\lambda}(t):= e^{-tm\sqrt{\lambda R}} \sum_{j=0}^{m-i} \frac{(tm\sqrt{\lambda R})^j}{j!}.&
   \end{align*}

\end{lemma}
\begin{proof}
\begin{align*}
\|e^{-xr-yt} \Delta_{\lambda}^{+}(x,y,A)\|_{m\sqrt{\lambda R}}^{m\sqrt{\lambda R}} &= \int_{x= \sqrt{R}}^{\infty} \int_{y=\sqrt{R}}^{\infty} e^{-xrm\sqrt{\lambda}-ytm\sqrt{\lambda}}y^{m-d}p_A(xy)dxdy \\
&=\int_{x= \sqrt{R}}^{\infty} \int_{y=\sqrt{R}}^{\infty} e^{-xrm\sqrt{\lambda}-ytm\sqrt{\lambda}} \mathbb{E}(e^{-T^{(m,d)}_{\Strans{p_A}} \partial{x}\partial_{y}}\{y^mx^d\}) dxdy \\
&= \sum_{i=0}^d\int_{x= \sqrt{R}}^{\infty} \int_{y=\sqrt{R}}^{\infty} \frac{ \mathbb{E}[(-T^{(m,d)}_{\Strans{p_A}})^{i}] }{i!} e^{-xrm\sqrt{\lambda}-ytm\sqrt{\lambda}} \frac{m!d!}{(m-i)!(d-i)!}y^{m-i}x^{d-i} dx dy.
\end{align*}
Now we use the fact that for $n$ integer and $a,b>0$,
\[
\int_{b}^{\infty} x^n e^{- ax} dx= \frac{n!}{a^{n+1}} e^{-ab} \sum_{i=0}^n \frac{(ab)^i}{i!}.
\]
We therefore get that
\begin{align*}
 \int_{x= \sqrt{R}}^{\infty} \int_{y=\sqrt{R}}^{\infty} \frac{ \mathbb{E}[(-T^{(m,d)}_{\Strans{p_A}})^{i} ]}{i!} e^{-xrm\sqrt{\lambda}-ytm\sqrt{\lambda}} \frac{m!d!}{(m-i)!(d-i)!}y^{m-i}x^{d-i} dx dy \\
 = \frac{ \mathbb{E}[(-T^{(m,d)}_{\Strans{p_A}})^{i}]}{i!}  \frac{m! d! L_i^{d,\lambda}(r) M_i^{d,\lambda}(t)}{(r\sqrt{\lambda R}m)^{d-i+1}(t\sqrt{\lambda R}m)^{m-i+1}}.
  \end{align*}
\end{proof}

\begin{cor} [Analycity of finite transforms ]\label{analyQ}
$s^2 \mathcal{Q}_{\mu_A}^{d,\lambda}(s)$ is analytic in $s$. As a consequence, $\widetilde{\mathcal{R}}^{d,\lambda}_{\mu_A}(s)$  is also analytic. 
\end{cor}
\begin{proof}
Let's rewrite it as a product of two terms and show that any of the two is a convergent power series in $s$:
\[
s^2 \mathcal{Q}_{\mu_A}^{d,\lambda}(s)=  r  {\frac{\partial}{\partial r}\big( \log 
 \|e^{-xr-yt} \Delta_{\lambda}^{+}(x,y,\mu_A)\|_{\sqrt{md}} \big)  
 \times t  \frac{\partial}{\partial t} \big(\log\| e^{-xr-yt} 
 \Delta_{\lambda}^{+}(x,y,\mu_A)\|_{\sqrt{md}} \big)}_{\big|_{r=r(s), t=t(s)}}.
\]
Using Lemma \ref{integral}, we obtain
\begin{align*}
 & \frac{\partial}{\partial r}\big( \log 
 \|e^{-xr-yt} \Delta_{\lambda}^{+}(x,y,\mu_A)\|_{m\sqrt{\lambda}} \big)  \\
 &= \frac{1}{m\sqrt{\lambda}} \frac{\partial}{\partial r} \log \Bigg(   \frac{m! d!}{(rm \sqrt{\lambda})^{d+1} (tm\sqrt{\lambda})^{m+1}}\sum_{i=0}^d \frac{\mathbb{E}[(-T^{(m,d)}_{\Strans{p_A}}rtmd)^i]}{i!}L_i^{d,\lambda}(r) M_i^{d,\lambda}(t)  \Bigg)\\
 &= \frac{1}{m\sqrt{\lambda}} \frac{\partial}{\partial r} \log \Big(   \frac{m! d!}{(rm \sqrt{\lambda})^{d+1} (tm\sqrt{\lambda})^{m+1}} \Big) + \frac{1}{m\sqrt{\lambda}} \frac{\partial}{\partial r} \log \Big(\sum_{i=0}^d \frac{\mathbb{E}[(-T^{(m,d)}_{\Strans{p_A}}rtmd)^i]}{i!}L_i^{d,\lambda}(r) M_i^{d,\lambda}(t)  \Big)\\
 &=  \frac{1}{m\sqrt{\lambda}}  \frac{-(d+1)}{r} +  \frac{1}{m\sqrt{\lambda}} \frac{\partial}{\partial r} \log \Big( 1 + \sum_{i=1}^d \frac{\mathbb{E}[(-T^{(m,d)}_{\Strans{p_A}}rtmd)^i]}{i!}L_i^{d,\lambda}(r) M_i^{d,\lambda}(t)  \Big).
 \end{align*}
 Therefore,
 \[
 r  \frac{\partial}{\partial r}\big( \log 
 \|e^{-xr-yt} \Delta_{\lambda}^{+}(x,y,\mu_A)\|_{m\sqrt{\lambda}} \big) =  -\frac{1}{m\sqrt{\lambda}} (d+1) + r \frac{\partial}{\partial r} \log \Big( 1 + \sum_{i=1}^d \frac{\mathbb{E}[(-T^{(m,d)}_{\Strans{p_A}}rtmd)^i]}{i!}L_i^{d,\lambda}(r) M_i^{d,\lambda}(t)  \Big).
 \]
 If we call
 \[
  \sum_{i,j= 1}^{\infty}a_{i,j}r^it^j:= \sum_{i=1}^d \frac{\mathbb{E}[(-T^{(m,d)}_{\Strans{p_A}}rtmd)^i]}{i!}L_i^{d,\lambda}(r) M_i^{d,\lambda}(t),
  \]
  which is analytic in two variables (see \cite{analy} for an introduction to multivariate complex analysis), we have that $\log \big( 1 + \sum_{i,j= 1}^{\infty}a_{i,j}r^it^j \big) = \sum_{i,j= 1}^{\infty}b_{i,j}r^it^j$ for $r$ and $t$ small enough and some coefficients $b_{i,j}$,  by composition of analytic functions, and also $r \frac{\partial}{\partial r} \log \big( 1+ \sum_{i,j= 1}^{\infty}a_{i,j}r^it^j \big)=  \sum_{i,j= 1}^{\infty}ib_{i,j}r^it^j$. We conclude by plugging in $r=r_{p_A}(s)$ and $t=t_{p_A}(s)$ which are both analytic which leads us to the analyticity of  $r  \frac{\partial}{\partial r}\big( \log 
 \|e^{-xr-yt} \Delta_{\lambda}^{+}(x,y,\mu_A)\|_{m\sqrt{\lambda}} \big)_{\big|_{r=r(s), t=t(s)}} $, which is what was required. 
\end{proof}

\subsubsection{Truncating the superfluous terms}
The expression we get for the $L_{m\sqrt{\lambda}}$-norms depend on $r$ and $t$ independently. We want to get rid of this bi-dependence and keep only a dependence in $rt=s^2$.\begin{lemma} \label{trunc}
\[
\sum _{i=0}^d\frac{\mathbb{E}[(-T^{(m,d)}_{\Strans{p_A}}rtmd)^i]}{i!}L_i^{d,\lambda}(r) M_i^{d,\lambda}(t) \equiv \sum_{i=0}^{d} \frac{\mathbb{E}[(-T^{(m,d)}_{\Strans{p_A}}rtmd)^i]}{i!} \mod [<r,t>^{d+1}].
\]
\end{lemma}

\begin{proof}
We need to prove that we can get rid of the  $L_i^{d,\lambda}(r)$,$M_i^{d,\lambda}(t)$ factors.
\[
e^{-rm\sqrt{\lambda}} \sum_{l=0}^{d-i} \frac{(rm\sqrt{\lambda})^l}{l!}= e^{-rm\sqrt{\lambda}} (e^{rm\sqrt{\lambda}} - \sum_{l=d-i+1}^{\infty} \frac{(rm\sqrt{\lambda})^l}{l!} ) \equiv 1 \mod [<r,t>^{d-i+1}].
\]
Similarly,
\[
e^{-tm\sqrt{\lambda}} \sum_{l=0}^{m-i} \frac{(tm\sqrt{\lambda})^l}{l!} \equiv 1 \mod [<r,t>^{m-i+1}] \equiv 1 \mod [<r,t>^{d-i+1}].
\]
It follows that
\begin{align*}
 L_i^{d,\lambda}(r) &=M_i^{d,\lambda}(t) \equiv 1 \mod [<r,t>^{d-i+1}]     \text{,}&      L_i^{d,\lambda}(r)M_i^{d,\lambda}(t)& \equiv 1 \mod [<r,t>^{d-i+1}].
 \end{align*}

And finally,
\begin{align*}
 \sum _{i=0}^{d}\frac{\mathbb{E}[(-T^{(m,d)}_{\Strans{p_A}}rtmd)^i]}{i!}L_i^{d,\lambda}(r) M_i^{d,\lambda}(t)  &\equiv  \sum _{i=0}^{d} \Big[\frac{\mathbb{E}[(-T^{(m,d)}_{\Strans{p_A}}rtmd)^i]}{i!} . \big( 1 \mod[<r,t>^{d-i+1}]\big) \Big]\\
  &\equiv  \sum _{i=0}^{d}\frac{\mathbb{E}[(-T^{(m,d)}_{\Strans{p_A}}rtmd)^i]}{i!} \mod[<r,t>^{d+1}],
\end{align*}
where we used that $(rt)^i  \big( 0 \mod [<r,t>^{d-i+1}] \big) \equiv 0 \mod [<r,t>^{d+1}]$.
\end{proof}

 \begin{lemma}
\[
 r\frac{\partial}{\partial r}\big( \log \|e^{-xr-yt} 
 \Delta_{\lambda}^{+}(x,y,\mu_A)\|_{\sqrt{md}} \big) \equiv 
 \frac{1}{\sqrt{md}} \Big( -(d+1)+ rt\frac{d}{dx} 
 \mathcal{E}^{d,\lambda}_{\mu_A}(x) \Big|_{x=rt}  \Big)
 \mod [<r,t>^{d+1}] \]
 \[
  t\frac{\partial}{\partial t} \big(\log\| e^{-xr-yt} 
  \Delta_{\lambda}^{+}(x,y,\mu_A)\|_{\sqrt{md}} \big)  \equiv 
  \frac{1}{\sqrt{md}} \Big( -(m+1)+rt\frac{d}{dx} 
  \mathcal{E}^{d,\lambda}_{\mu_A}(x) \Big|_{x=rt}  \Big)
 \mod [<r,t>^{d+1}]\]
where
 \[
 \mathcal{E}^{d,\lambda}_{\mu_A}(x)= \log \big( \mathbb{E} \sum_{i=0}^{d} \frac {(-T^{(m,d)}_{\Strans{p_A}}xmd)^i}{i!}  \big).
 \]
 
 \end{lemma}
 
 In the following we will denote by : $[\mathcal{E}{^{d,\lambda}_{\mu_A}}]'(u):= 
\frac{d}{dx}\mathcal{E}^{d,\lambda}_{\mu_A}(x)  \Big|_{x=u}$.

 \begin{proof}
 \begin{align}
& \frac{\partial}{\partial r}\big( \log \|e^{-xr-yt} 
 \Delta_{\lambda}^{+}(x,y,\mu_A)\|_{\sqrt{md}} \big) = \frac{1}{m\sqrt{\lambda}}  
 \frac{\partial}{\partial r}\big( \log \|e^{-xr-yt} 
 \Delta_{\lambda}^{+}(x,y,\mu_A)\|_{m\sqrt{\lambda}}^{m\sqrt{\lambda} }\big) \\
 &= \frac{1}{m\sqrt{\lambda}} \Big( -\frac{d+1}{r} +\frac{\partial}{\partial r} \log\big( \sum _{i=0}^d\frac{\mathbb{E}[(-T^{(m,d)}_{\Strans{p_A}}rtmd)^i]}{i!}L_i^{d,\lambda}(r) M_i^{d,\lambda}(t)   \big)   \Big).
\end{align}

Then,
 \begin{align}
&r \frac{\partial}{\partial r}\big( \log \|e^{-xr-yt} 
 \Delta_{\lambda}^{+}(x,y,\mu_A)\|_{\sqrt{md}} \big)= \frac{1}{m\sqrt{\lambda}} \Big( -(d+1)+r \frac{\partial}{\partial r} \log\big[\sum _{i=0}^d\frac{\mathbb{E}[(-T^{(m,d)}_{\Strans{p_A}}rtmd)^i]}{i!}L_i^{d,\lambda}(r) M_i^{d,\lambda}(t)   \big]   \Big)\\
 &\equiv \frac{1}{m\sqrt{\lambda}} \Big( -(d+1)+r \frac{\partial}{\partial r}\log\Big \{\sum _{i=0}^{d}\frac{\mathbb{E}[(-T^{(m,d)}_{\Strans{p_A}}rtmd)^i]}{i!}  \mod[<r,t>^{d+1}] \Big \}  \Big)     \text{  (Using Lemma \ref{trunc}})\\
 &\equiv  \frac{1}{m\sqrt{\lambda}} \Big( -(d+1)+r \frac{\partial}{\partial r} \log\Big \{\sum _{i=0}^{d}\frac{\mathbb{E}[(-T^{(m,d)}_{\Strans{p_A}}rtmd)^i]}{i!}  \Big \}\mod[<r,t>^{d+1}]  \Big)
\end{align}
where we used that applying a power series ($\log$) or the operator $r\partial_r$ doesn't affect a modulo equality, which follows from \ref{truncation}.
Now we conclude by noticing that
\[
\frac{\partial}{\partial r} \mathcal{E}^{d,\lambda}_{\mu_A}(rt)= t[\mathcal{E}^{d,\lambda}_{\mu_A}]'(rt).
\]

\begin{rmk} \label{remarkpos}
One could wonder if the logarithm defined above is well defined, as we are not sure whether the quantity inside is positive. It follows from the fact that the norm is always positive and that for $r$ and $t$ small enough (using the asymptotic expansion in $t$ and $r$),
\[
\sum _{i=0}^d\frac{\mathbb{E}[(-T^{(m,d)}_{\Strans{p_A}}rtmd)^i]}{i!}L_i^{d,\lambda}(r) M_i^{d,\lambda}(t) >0   \implies \sum_{i=0}^{d} \frac{\mathbb{E}[(-T^{(m,d)}_{\Strans{p_A}}rtmd)^i]}{i!} >0.
\]
\end{rmk}
\end{proof}

\subsubsection{Back to univariate power series}

 \begin{cor}
 \[
s^2 \mathcal{Q}_{\mu_A}^{d,\lambda}(s) \equiv \frac{1}{md}\big[(m+1)- s^2[\mathcal{E}^{d,\lambda}_{\mu_A}]'(s^2)   \big]\big[(d+1)-s^2[\mathcal{E}^{d,\lambda}_{\mu_A}]'(s^2)   \big]  \mod[s^{d+1}].
\]
\end{cor}
 \begin{proof}
 \begin{align*}
& r \frac{\partial}{\partial r}\big( \log \|e^{-xr-yt} 
 \Delta_{\lambda}^{+}(x,y,\mu_A)\|_{\sqrt{md}} \big) \times t \frac{\partial}{\partial t}\big( \log \|e^{-xr-yt} 
 \Delta_{\lambda}^{+}(x,y,\mu_A)\|_{\sqrt{md}} \big) \equiv \\
 &\frac{1}{\sqrt{md}} \Big( -(d+1)+ rt\frac{d}{dx} 
 \mathcal{E}^{d,\lambda}_{\mu_A}(x) \Big|_{x=rt}  \Big) \times \frac{1}{\sqrt{md}} \Big( -(m+1)+rt\frac{d}{dx} 
  \mathcal{E}^{d,\lambda}_{\mu_A}(x) \Big|_{x=rt}  \Big)  \mod[<r,t>^{d+1}].
  \end{align*}
  Now plug in $r=r_{p_A}(s)$ and $t=t_{p_A}(s)$ and notice that as $r_{p_A}(s)= s+ \sum_{i=2}^{\infty}r_is^i$ and $t_{p_A}(s)= s+ \sum_{i=2}^{\infty}t_is^i$,
  \[ f \equiv g \mod[<r_{p_A}(s),t_{p_A}(s)>^{d+1}]  \implies  f \equiv g \mod[s^{d+1}] \].

 \end{proof}
 
 \begin{cor}
 \[
 \sqrt{\frac{1}{\lambda} s^2 \mathcal{Q}_{\mu_A}^{d,\lambda}(s) + \frac{(1-\lambda)^2}{4\lambda^2}   } \equiv \frac{1}{d}\Big(-s^2[\mathcal{E}^{d,\lambda}_{\mu_A}]'(s^2)+ \big[ m\frac{(\lambda+1)}{2}+1\big] \Big)   \mod[s^{d+1}].
 \]
 \end{cor}
 
 \begin{proof}
\begin{align}
&s^2  \mathcal{Q}_{\mu_A}^{d,\lambda}(s)  \equiv \frac{1}{m^2\lambda}\Big(\big[ s^2[\mathcal{E}^{d,\lambda}_{\mu_A}]'(s^2)\big]^2 - s^2[\mathcal{E}^{d,\lambda}_{\mu_A}]'(s^2)\big[m+1+m\lambda+1\big]\Big) +\frac{(m+1)(m\lambda +1)}{m^2\lambda}  \mod[s^{d+1}] \\
& \equiv \frac{1}{m^2\lambda} \Big(-s^2[\mathcal{E}^{d,\lambda}_{\mu_A}]'(s^2)+ \big[m\frac{(\lambda+1)}{2}+1\big] \Big)^2 + \frac{-\Big( m\frac{(\lambda+1)}{2}+1\Big)^2+(m+1)(m\lambda+1)}{m^2\lambda}  \mod[s^{d+1}] \\
& \equiv \frac{1}{m^2\lambda} \Big(-s^2[\mathcal{E}^{d,\lambda}_{\mu_A}]'(s^2)+ \big[m\frac{(\lambda+1)}{2}+1\big] \Big)^2 - \frac{(\lambda-1)^2}{4\lambda}  \mod[s^{d+1}].
\end{align}

Therefore, multiplying both sides of the equation by $\frac{1}{\lambda} $ and adding $\frac{(\lambda-1)^2}{4\lambda^2} $: 
\[
\frac{1}{\lambda} s^2\mathcal{Q}_{\mu_A}^{d,\lambda}(s) + \frac{(\lambda-1)^2}{4\lambda^2} \equiv \frac{1}{m^2 \lambda^2}\Big(-s^2[\mathcal{E}^{d,\lambda}_{\mu_A}]'(s^2)+ \big[m\frac{(\lambda+1)}{2}+1\big] \Big)^2    \mod[s^{d+1}].
\]
As $\Big(-s^2[\mathcal{E}^{d,\lambda}_{\mu_A}]'(s^2)+ \big[m\frac{(\lambda+1)}{2}+1\big] \Big) \geq 0$, for $s$ small enough, we can take the square root of our modulo equality (using the stability of the modulo equalities with respect to composition with the analytic function $\sqrt{1+s}$ near $0$) :
\begin{align*}
\sqrt{\frac{1}{\lambda} s^2\mathcal{Q}_{\mu_A}^{d,\lambda}(s) + \frac{(\lambda-1)^2}{4\lambda^2}} & \equiv \sqrt{  \frac{1}{m^2 \lambda^2}\Big(-s^2[\mathcal{E}^{d,\lambda}_{\mu_A}]'(s^2)+ \big[m\frac{(\lambda+1)}{2}+1\big] \Big)^2 }  \mod[s^{d+1}]  \\
& \equiv \frac{1}{d}\Big(-s^2[\mathcal{E}^{d,\lambda}_{\mu_A}]'(s^2)+ \big[m\frac{(\lambda+1)}{2}+1\big] \Big)  \mod[s^{d+1}].
\end{align*}
We finally obtain
\[
\sqrt{\frac{1}{\lambda} s^2\mathcal{Q}_{\mu_A}^{d,\lambda}(s) + \frac{(\lambda-1)^2}{4\lambda^2}}  - \Big(\frac{[\lambda+1]}{2\lambda}+\frac{1}{d}\Big) \equiv \frac{1}{d}\Big(-s^2[\mathcal{E}^{d,\lambda}_{\mu_A}]'(s^2) \Big)  \mod[s^{d+1}].
\]
\end{proof}
 \begin{thm}\label{partialequality}
 \[
 \widetilde{\mathcal{R}}^{d,\lambda}_{\mu_A}(s)  \equiv \mathcal{R}^{d,\lambda}_{\mu_A}(s^2)   \mod [ s^{d+1}]
 \]
 \end{thm}
 \begin{proof}
 The previous equality gives us
 \[
  \widetilde{\mathcal{R}}^{d,\lambda}_{\mu_A}(s) \equiv  \frac{1}{d}\Big(-s^2[\mathcal{E}^{d,\lambda}_{\mu_A}]'(s^2) \Big) \mod [ s^{d+1}].
 \]
 On the other hand,
  \begin{align*}& \mathbb{E}\big[ e^{-T^{(m,d)}_{\Strans{p_A}}s^2md}\big] = \sum_{i=0}^{\infty}  \frac{  \mathbb{E}\big[-T^{(m,d)}_{\Strans{p_A}}s^2md\big]^i}{i!} 
 \equiv \sum_{i=0}^{d}  \frac{  \mathbb{E}\big[-T^{(m,d)}_{\Strans{p_A}}s^2md\big]^i}{i!}  \mod [ s^{d+1}]\\
 & \log \big(\mathbb{E}\big[ e^{-T^{(m,d)}_{\Strans{p_A}}s^2md}\big] \big)   \equiv  \log \big( \sum_{i=0}^{d}  \frac{  \mathbb{E}\big[-T^{(m,d)}_{\Strans{p_A}}s^2md\big]^i}{i!} \big) \mod [ s^{d+1}] \\
 & \frac{d}{ ds}\log \big(\mathbb{E}\big[ e^{-T^{(m,d)}_{\Strans{p_A}}s^2md}\big] \big)   \equiv  \frac{d}{ds} \log \big( \sum_{i=0}^{d}  \frac{  \mathbb{E}\big[-T^{(m,d)}_{\Strans{p_A}}s^2md\big]^i}{i!} \big) \mod [ s^{d}] \\
  & \frac{-s}{d}\frac{d}{ds}\log \big(\mathbb{E}\big[ e^{-T^{(m,d)}_{\Strans{p_A}}s^2md}\big] \big)   \equiv \frac{-s}{d} \frac{d}{d s} \log \big( \sum_{i=0}^{d}  \frac{  \mathbb{E}\big[-T^{(m,d)}_{\Strans{p_A}}s^2md\big]^i}{i!} \big) \mod [ s^{d+1}] \\
  &  \mathcal{R}^{d,\lambda}_{\mu_A}(s^2) \equiv \frac{1}{d}\Big(-s^2[\mathcal{E}^{d,\lambda}_{\mu_A}]'(s^2) \Big) \mod [ s^{d+1}].
 \end{align*}

 \end{proof}
 \begin{rmk}
 We notice that the degree $d$ truncation of   $\tilde{\mathcal{R}}^{d,\lambda}_{\mu_A}(s) $ is in fact a polynomial in $s^2$. Only high order odd terms can be nonzero. 
  \end{rmk}

\section{From finite free probability to free probability: a bridge between algebra and analysis}\label{sec:convergence}
\begin{thm}(Convergence of the finite free rectangular transform to the free transform)\label{convergence}
The following convergence holds coefficientwise and also pointwise for $s$ small enough: 
\[
 \mathcal{R}^{d,\lambda}_{\mu_A}(s) \xrightarrow{d\to \infty}   \mathcal{R}^{\lambda}_{\mu_A}(s).
  \]
\end{thm}
\begin{proof}
Let's write
\begin{align*}
 \mathcal{R}^{\lambda}_{\mu_A}(s)&:= \sum_k a_ks^k    &    \mathcal{R}^{d,\lambda}_{\mu_A}(s)&:= \sum_k a^d_ks^k  & \widetilde{\mathcal{R}}^{d,\lambda}_{\mu_A}(s)&:= \sum_k \tilde{a}^d_ks^k.
  \end{align*}
  By Theorem~\ref{Convmodified}, we get the convergence for all $s$ small enough of  $\widetilde{\mathcal{R}}^{d,\lambda}_{\mu_A}(s)$ to $ \mathcal{R}^{\lambda}_{\mu_A}(s^2)$. In particular we get (as convergence over an interval gives us the convergence of all coefficients separately) that
  \begin{align*}
   \tilde{a}^d_{2k+1}  & \xrightarrow{d\to \infty}  0  &       \tilde{a}^d_{2k}  &  \xrightarrow{d\to \infty}  a_k.
     \end{align*}
     We also get from Corollary~\ref{partialequality} that  $\tilde{a}^d_{2k}= a^d_k$ for $k \leq d/2$, which in particular gives us that $a^d_k \to_{d \rightarrow \infty} a_k$ for all $k$. This doesn't mean that    $\mathcal{R}^{d,\lambda}_{\mu_A}(s)$ converges to  $\mathcal{R}^{\lambda}_{\mu_A}(s)$ or does have a limit when $d$ goes to infinity.
     
       \begin{lemma}\label{tailemma}
  For $s$ small enough, $\mathcal{R}^{d,\lambda}_{\mu_A}(s)$ has a tail of order $d/2$ that goes to zero when $d$ goes to infinity.
  \end{lemma}
     \begin{proof}
     Let's recall that $p_A(x)= \sum_{i}(-1)^ix^{d-i}p_i$ and that
     \[
        \frac{ \mathbb{E} \Big(\big[T^{(m,d)}_{\Strans{p_A}}md \big]^i \Big)}{i!}=   \frac{m^i(m-i)!}{m!}\frac{d^i(d-i)!}{d!} p_i.
     \]
     As the roots are uniformly bounded by $R$, we get that $|p_i| \leq R^i \binom{d}{i} \leq R^i \binom{d}{d/2}$.
      On the other hand, 
          \begin{align*}
          \binom{d}{d/2} &\leq  2^{d}    &  \frac{d^i(d-i)!}{d!} &\leq  \frac{d^i}{(d/2)^i} \leq 2^d  &  \frac{m^i(m-i)!}{m!} &\leq 2^{\frac{d}{\lambda}}.
            \end{align*}
 Grouping the upper bounds we get for $ \alpha>0$ and  all $i \leq d$,
 \[
 \Big|   \frac{ \mathbb{E}\Big(\big[-T^{(m,d)}_{\Strans{p_A}}md \big]^i \Big)}{i!} \Big| \leq e^{d\alpha}.
 \]
 We want to investigate the expansion of
    \[
     \mathcal{R}^{\lambda}_{\mu_A}(s)=   \frac{-1}{d}s \frac{d}{ds}\log \big( \mathbb{E}\big[ e^{-T^{(m,d)}_{\Strans{p_A}}smd}\big]  \big) \mod[s^{d+1}].
  \]
 Now when we look at the coefficients in the expansion in $s$ of $\log \big( \mathbb{E}\big[ e^{-T^{(m,d)}_{\Strans{p_A}}smd}\big]  \big) $. We get $p(k)$ terms of order $k$ expanding the series  where $p(k)$ is the number of partitions of $k$. Overall this number can be upper bounded by $e^{\beta k}$ for some $\beta >0$ ( see \cite{wikipartition}), so that the $k^{th}$ coefficient in the expansion will be upper bounded by $e^{(\alpha+\beta) d}$, for $ k \geq d/2$. We conclude as 
 \[ 
 \sum_{k=d/2}^d |a_{k}^ds^k| \leq  \sum_{k=d/2}^d k |s|^k e^{(\alpha+\beta) k},
 \]
 qnd the last sum goes to zero for $s$ small enough.
       
     \end{proof}
     Back to the convergence. We readily get from Corollary~\ref{partialequality} that
     \begin{equation}\label{convdecompo}
       \mathcal{R}^{d,\lambda}_{\mu_A}(s)=    (\widetilde{\mathcal{R}}^{d,\lambda}_{\mu_A}(s) \mod [s^{d/2}] ) +  \sum_{k=d/2}^d a_{k}^ds^k.
     \end{equation}
     As $\widetilde{\mathcal{R}}^{d,\lambda}_{\mu_A}(s)$ converges to $\mathcal{R}^{\lambda}_{\mu_A}(s)$ for $s$ small enough, then its partial sum of order $d/2$,  the first term in \ref{convdecompo}, also converges to this same limit. Using Lemma~\ref{tailemma}, the second term in \ref{convdecompo} goes to zero and finally we get that   $\mathcal{R}^{d,\lambda}_{\mu_A}(s)$ converges to $\mathcal{R}^{\lambda}_{\mu_A}(s)$ .
      \end{proof}.

\begin{lemma}
  A good way to exhibit actual matrices/polynomials that have constant underlying measure and whose dimensions go to infinity is to stack up (by block) $n$ identical matrices $A$ of fixed size $m \times d$(associated to a polynomial $p$) along the diagonal, and their  characteristic polynomials will be $p^n$. We can make $n$ go to infinity and we get
\[
 \mathcal{R}^{dn,\lambda}_{\Strans{p}^n}(s) \xrightarrow{n\to \infty}  \Rtransform{\Strans{p}} (s).
\]
\end{lemma}
\begin{proof}
Consider $p^n$, a polynomial with roots repeated $n$ times. It is a way to keep the measure on the roots constant and increase the size of the matrix. We have
\[
\mu_{\Strans{p}}= \mu_{\Strans{p^n}},
\]
and using Theorem \ref{convergence} we get that
\[
 \mathcal{R}^{dn,\lambda}_{\Strans{p^n}}(s) \xrightarrow{n\to \infty}   \Rtransform{\Strans{p}} (s).
\]
\end{proof}
\begin{cor}
\[
\mathcal{R}^{dn,\lambda}_{\Strans{[p^n \recsum q^n]}}(s)   \xrightarrow{n\to \infty} \Rtransform{\mu_{\Strans{p}} \recsumfree \mu_{\Strans{q}}} (s).
\]
\end{cor}

\begin{proof}
We have on the one hand
\[
 \mathcal{R}^{dn,\lambda}_{\Strans{[p^n \recsum q^n]}}(s) = \mathcal{R}^{dn,\lambda}_{\Strans{p^n}}(s) +\mathcal{R}^{dn,\lambda}_{\Strans{q^n}}(s),
\]
so that
\[
    \mathcal{R}^{dn,\lambda}_{\Strans{[p^n \recsum q^n]}}(s)  \xrightarrow{n\to \infty}  \Rtransform{\Strans{p}} (s) + \Rtransform{q} (s) = \Rtransform{ \mu_{\Strans{p}}\recsumfree \mu_{\Strans{q}}} (s).
    \]

\end{proof}

\section{Limit theorems and special polynomials}\label{limittheorems}
In this section, we try to show how we can actually consider even realrooted polynomials as equivalent of independent random variables in our framework.
Recall that, for a real rooted polynomial $r$ of degree $d$ with roots $r_i$, we can define its expectation and variance as
\begin{align*}
\mathbb{E}(r)&:=\frac{1}{d}\sum_i  r_i    &\text{and}& &    Var(r)&:= \frac{1}{d}\sum_i  (r_i- \mathbb{E}(r))^2.
\end{align*}

\begin{lemma}\label{explemma} Exactly like for a sum of independent random variables, we get
\begin{itemize}
\item  $\mathbb{E}\big[\Strans{[p \boxplus_{d,\lambda} q]}\big]= \mathbb{E}[\Strans{p}] +\mathbb{E}[\Strans{q}] =0$ , 

\item $Var\big[\Strans{[p \boxplus_{d,\lambda}q]}\big]=Var[\Strans{p}]+Var[\Strans{q}]$.
   
\end{itemize}
\end{lemma}
\begin{proof}
It follows by inspection of the two first coefficients (associated to $1$ and $x$) in the finite $R$-transform. 
\end{proof}

We study in this section asymptotics related to the convolution. The only single-rooted symmetric polynomial is the zero polynomial $x^d$. The law of large numbers in probability states that for random variables that have the same mean and are independent, then the average converges to this mean, so the constant random variable. 
We adapt it in this context: we add polynomials "freely" and renormalize the symmetrization of their sum. The result is that it converges to the zero polynomial.
Let's first introduce a renormalization operator that rescales the roots and establish a few basic properties verified by $R$- transforms.
For $p$ of degree $d$, and $\alpha>0$, define $\mathit{R_{\alpha}}(p):= \alpha^{-d}p(\alpha x)$.
 For two polynomials $p$ and $q$, we will use $p \approx q$ to express that they have the same roots but not the same leading coefficient.
\begin{lemma}\label{operations}
For all $\alpha>0$,
$\mathit{R}_{\alpha}\Strans{p} = \Strans{[\mathit{R}_{\alpha^2}p]}$. Also, $\mathit{R}_{\alpha}[p \boxplus_{d,\lambda} q] =\mathit{R}_{\alpha}p \boxplus_{d,\lambda} \mathit{R}_{\alpha}q$. 
\end{lemma}
\begin{proof}
Straightforward using the definitions. 
\end{proof}

\begin{lemma}\label{Rtransformscaling}
 If   $\mathcal{R}^{d,\lambda}_{\Strans{p}}(s)= \sum_{k=1}^db_ks^k$, then
  $\mathcal{R}^{d,\lambda}_{\Strans{\mathit{R}_{\alpha}p}}(s)=  \sum_{k=1}^db_k\alpha^ks^k$. 
\end{lemma}
\begin{proof}
Using the fact that $\mathit{R}_{\alpha}$ multiplies the roots by $\alpha$ and therefore rescales the coefficients of $p$, $p_i$, by  $\alpha^i$, we get
\[
 \mathbb{E}\Big( [T^{(m,d)}_{\Strans{\mathit{R}_{\alpha}p}}]^i\Big)=  \mathbb{E}\Big( [T^{(m,d)}_{\Strans{p}}]^i \Big) \alpha^i.
 \]
 We then get expanding and using linearity,
 \[
 \mathcal{R}^{d,\lambda}_{\Strans{R_{\alpha}p}}(s) =  \sum_{k=1}^db_k \alpha^ks^k.
 \]
\end{proof}

\begin{lemma}\label{Rboundedness}
Fix $d$. Then if  $p_1,p_2$,... are a sequence of degree $d$ polynomials with real nonnegative roots and means uniformly bounded, then all the coefficients are uniformly bounded too, and the coefficients of the  $\mathcal{R}^{d,\lambda}_{ \Strans{p_i}}$ as well. By uniform bound, we mean that it is valid for all polynomials with this property. 
\end{lemma}
\begin{proof}
The first part follows from the $k$-norms of roots being all equivalent and therefore bounded by the $1$-norm (the parameter depending only on the dimensions), when the size $d$ is fix (and $k \in [|1,d|]$), and the coefficients of the polynomials being combinations of the $k$-norms of the roots. It could also be proven using repeatedly Cauchy Schwarz inequality. As for the second part, notice that if $d$ and $m$ are fix, then the coefficients of $\mathcal{R}^{d,\lambda}_{ \Strans{p_i}}$  are polynomial in  the coefficients of the polynomial $p_i$, that are uniformly bounded as we just saw, so the uniform boundedness of the coefficients of $\mathcal{R}^{d,\lambda}_{ \Strans{p_i}}$  follows. 
\end{proof}
\begin{prop}[Law of large numbers]  \label{lawlargenumbers} Let $p_1,p_2$,... be a sequence of degree $d$ polynomials with real nonnegative roots and means uniformly bounded by $\sigma^2$, that is,
\begin{align*}  & p_i:= \prod_j (x-r_{i,j}^2) & &\text{and}& \frac{1}{d}\sum_j r_{i,j}^2 \leq \sigma^2 &,
\end{align*}  
then
\[
\lim_{N \to \infty} \mathit{R_{1/{N}}}(\Strans{[p_1\boxplus_{d,\lambda}... \boxplus_{d,\lambda} p_N]}) (x)\approx  x^d.
\]
\end{prop}
\begin{proof}
Let's write for all $i$, 
 \[
 \mathcal{R}^{d,\lambda}_{ \Strans{p_i}}(s):=   \sum_{k=1}^db_{k,i}  s^k.
 \]
Using successively Lemma~\ref{Rtransformscaling} and Lemma~\ref{operations},
\begin{align*}
 \mathcal{R}^{d,\lambda}_{ \mathit{R}_{1/{N}}(\Strans{[p_1\boxplus_{d,\lambda}... \boxplus_{d,\lambda} p_N]})}(s) &=   \mathcal{R}^{d,\lambda}_{ \Strans{ \mathit{R}_{1/{N^2}}[p_1\boxplus_{d,\lambda}... \boxplus_{d,\lambda} p_N]}}(s) \\
 &=  \mathcal{R}^{d,\lambda}_{ \Strans{\big[ (\mathit{R}_{1/{N^2}}p_1)\boxplus_{d,\lambda}... \boxplus_{d,\lambda} (\mathit{R}_{1/{N^2}}p_N)} \big]}(s) \\
 &= \sum_{i=1}^N \mathcal{R}^{d,\lambda}_{ \Strans{ \mathit{R}_{1/{N^2}}p_i}}(s)\\
 &=  \sum_{i=1}^N  \sum_{k=1}^db_{k,i}  \frac{1}{N^{2k}}s^k.
 \end{align*}
Due to Lemma~\ref{Rboundedness}, the uniform boundedness of the $\big(b_{k,i}\big)_{k,i}$ holds, and therefore there exists $K>0$ such that  $|b_{k,i}| \leq K$ for all $i \leq N$, $k\leq d$. As a consequence, we get $\frac{  |\sum_{i=1}^N b_{k,i} |}{ N^{2}} \leq \frac{K}{N}$. Therefore
\[
\lim_{N \to \infty} \mathcal{R}^{d,\lambda}_{ \mathit{R}_{1/{N}}(\Strans{[p_1\boxplus_{d,\lambda}... \boxplus_{d,\lambda} p_N]})}(s) = 0.
\]
Using the inversion formula through which one recovers the polynomial from its $R$-transform(see Proposition~\ref{inversionf}), it means, as the polynomial with $0$ $R$-transform is trivially $x^d$, that 
\[
\lim_{N \to \infty}\mathit{R}_{1/{N}}(\Strans{[p_1\boxplus_{d,\lambda}... \boxplus_{d,\lambda} p_N]}) = x^d.
\]

\end{proof}

\begin{lemma}[Laguerre $R$ and $T$-transforms] \label{laguerrecumulant}
For a polynomial $p$ of degree $d$ with nonnegative real roots, and $\sigma^2>0$, the following are equivalent:
\begin{enumerate}
\item $T^{(m,d)}_{\Strans{p}}=\sigma^2$, so the $T$-transform is constant (it has to be positive).
\item $\mathcal{R}^{d,\lambda}_{\Strans{p}} (s) =m\sigma^2 s$, so only the first nontrivial cumulant is nonzero.
\item $ p \approx L_d^{(m-d)}(\frac{x}{\sigma^2})$, that is $p$ is up to scaling a generalized Laguerre polynomial of parameter $m-d$.
\end{enumerate}
\end{lemma}
\begin{proof}
Assume that $T^{(m,d)}_{\Strans{p}}=\sigma^2$. Then
\begin{align*}
 \mathcal{R}^{d,\lambda}_{\Strans{p}}(s) & \equiv \frac{-1}{d}s \frac{d}{ds}\log \big( \mathbb{E} \big[e^{-T^{(m,d)}_{\Strans{p}}smd} \big] \big) \mod [ s^{d+1}]\\
\mathcal{R}^{d,\lambda}_{\Strans{p}}(s)  &= \sigma^2ms.
 \end{align*}
 
 We can use the bijection between transforms to conclude that the two first points are equivalent. We get the equivalence with the third point as
 \begin{align}
 e^{-\sigma^2 \partial{x}\partial_{y}}\{y^mx^d\} &= \sigma^{2d} y^{m-d}\sum_{i=0}^d \big(\frac{-xy}{\sigma^2}\big)^{d-i}\frac{m!d!}{i! (m-i)! (d-i)!} \\
 &= d! \sigma^{2d} y^{m-d}\sum_{i=0}^d \big( \frac{-xy}{\sigma^2}\big)^i\frac{m!}{i! (m-d+i)! (d-i)!}\\
 &= d!\sigma^{2d}y^{m-d}\sum_{i=0}^d  \big( \frac{-xy}{\sigma^2}\big)^i \frac{1}{i!} {m \choose d-i} \\ 
 &=  d!\sigma^{2d}y^{m-d} L_d^{(m-d)}(\frac{xy}{\sigma^2}).
 \end{align}

\end{proof}

\begin{cor}[Divisibility of the Laguerre polynomials]
\[
 L_d^{(m-d)}(\frac{x}{\sigma^2+\tau^2}) \approx  L_d^{(m-d)}(\frac{x}{\sigma^2}) \boxplus_{d,\lambda}   L_d^{(m-d)}(\frac{x}{\tau^2}).
 \]

\end{cor}

The central limit theorem in probability states that for independent random variables that have zero mean and constant variance, then the square-root average converges to a Gaussian random variable that has the same variance. We adapt it in this context: we add polynomials freely, and renormalize the roots by $\sqrt{N}$ . The result is that it converges to a generalized Laguerre polynomial, the equivalent of a Gaussian random variable. 
\begin{prop}[Central limit theorem]  \label{centrallimittheorem} Let $p_1,p_2$,... be a sequence of degree $d$ with real nonnegative roots and same mean $\sigma^2$, that is,
\begin{align*}  & p_i= \prod_j (x-r_{i,j}^2) &\text{and}&&  \frac{1}{d}\sum_j r_{i,j}^2=\sigma^2 &
\end{align*} 
then
\[
\lim_{N \to \infty} \mathit{R_{1/\sqrt{N}}}(\Strans{[p_1\boxplus_{d,\lambda}... \boxplus_{d,\lambda} p_N]}) (x)\approx  L_d^{(m-d)}(\frac{x^2m}{\sigma^2}).
\]
\end{prop}

\begin{proof}
Let's write again for all $i$, 
 \[
 \mathcal{R}^{d,\lambda}_{ \Strans{p_i}}(s):=   \sum_{k=1}^db_{k,i}  s^k.
 \]
Using successively Lemma~\ref{Rtransformscaling} and Lemma~\ref{operations},
\begin{align*}
 \mathcal{R}^{d,\lambda}_{ \mathit{R}_{1/{\sqrt{N}}}(\Strans{[p_1\boxplus_{d,\lambda}... \boxplus_{d,\lambda} p_N]})}(s) &=   \mathcal{R}^{d,\lambda}_{ \Strans{ \mathit{R}_{1/{N}}[p_1\boxplus_{d,\lambda}... \boxplus_{d,\lambda} p_N]}}(s) \\
 &=  \mathcal{R}^{d,\lambda}_{ \Strans{\big[ (\mathit{R}_{1/{N}}p_1)\boxplus_{d,\lambda}... \boxplus_{d,\lambda} (\mathit{R}_{1/{N}}p_N)} \big]}(s) \\
 &= \sum_{i=1}^N \mathcal{R}^{d,\lambda}_{ \Strans{ \mathit{R}_{1/{N}}p_i}}(s)\\
 &=  \sum_{i=1}^N  \sum_{k=1}^db_{k,i}  \frac{1}{N^{k}}s^k.
 \end{align*}
Due to Lemma~\ref{Rboundedness}, the uniform boundedness of the $\big(b_{k,i}\big)_{k,i}$ holds, and therefore there exists $K>0$ such that  $|b_{k,i}| \leq K$ for all $i \leq N$, $k\leq d$. 
We have for $k\geq 2$,
 \[
 \frac{  |\sum_{i=1}^N b_{k,i} |}{ N^{k}} \leq \frac{K}{N}.
\]
Therefore
\[
\lim_{N \to \infty} \sum_{i=1}^N  \sum_{k=2}^db_{k,i}  \frac{1}{N^{k}}s^k = 0.
\]
Notice that $a^{(d,i)}_1= m\sigma^2$ and as a consequence, 
\[
  \sum_{i=1}^N a^{(d,i)}_1  \frac{1}{N}s =  m\sigma^2 s.
\]
We then get
\[
\lim_{N \to \infty} \mathcal{R}^{d,\lambda}_{ \mathit{R}_{1/{N}}(\Strans{[p_1\boxplus_{d,\lambda}... \boxplus_{d,\lambda} p_N]})}(s) = m\sigma^2 s.
\]
Using the inversion formula through which one recovers the polynomial from its $R$-transform(see Proposition~\ref{inversionf}), it means
\[
\lim_{N \to \infty}\mathit{R}_{1/{N}}(\Strans{[p_1\boxplus_{d,\lambda}... \boxplus_{d,\lambda} p_N]})(s) \approx  L_d^{(m-d)}(\frac{x^2m}{\sigma^2}).
\]

\end{proof}

\section*{Acknowledgements}
\textit{I would like to thank my advisor Adam Marcus for suggesting this problem to me and taking time to discuss the details with me.}

\section*{Declarations}
This study was funded by Princeton University and EPFL (jointly). The author has no conflicts of interest to declare that are relevant to the content of this article.
Data sharing not applicable to this article as no datasets were generated or analysed during the current study.

\end{document}